\documentclass
[
11pt]
{amsart}%

\usepackage{hyperref}
\usepackage{cite}
\usepackage{amssymb}
\usepackage{amsmath}
\usepackage{amsthm}
\usepackage{esint}
\usepackage{amsfonts}
\usepackage{graphicx}
\usepackage{bbm}
\usepackage{xcolor}
\usepackage[toc,page]{appendix}
\usepackage{subfigure}
\usepackage{mathtools}
\usepackage{mathrsfs}
\usepackage{bm} 
\usepackage[normalem]{ulem}
\usepackage{comment}
\usepackage{tikz, tikz-cd}
\usepackage{breakcites}
\usetikzlibrary{matrix,arrows,decorations.pathmorphing}

\newcommand{\liangbing}[1]{{\color{red} Liangbing says: #1}}

\newtheorem{theorem}{Theorem}[section]

\newtheorem{notation}[theorem]{Notation}
\newtheorem{proposition}[theorem]{Proposition}
\newtheorem{corollary}[theorem]{Corollary}
\newtheorem{lemma}[theorem]{Lemma}
\newtheorem{remark}[theorem]{Remark}

\newtheorem*{acknowledgement}{Acknowledgement}
\theoremstyle{definition}
\newtheorem{definition}[theorem]{Definition}

\numberwithin{equation}{section}

\begin{document}

\title[Log Sobolev on homogeneous spaces]{Logarithmic Sobolev inequalities on homogeneous spaces}

\author[Maria Gordina]{Maria Gordina{$^{\dag}$}}
\thanks{\footnotemark {$\dag$} Research was supported in part by NSF grant DMS-2246549. The author acknowledges the support by the Hausdorff Center of Mathematics (Bonn, Germany) and the IHES (France), where parts of the work were completed.}
\address{$^{\dag}$ Department of Mathematics\\
University of Connecticut\\
Storrs, CT 06269,  U.S.A.}
\email{maria.gordina@uconn.edu}

\author[Liangbing Luo]{Liangbing Luo}
\address{Department of Mathematics\\
Lehigh University\\
Bethlehem, PA 18015,  U.S.A.}
\email{lil522@lehigh.edu}


\maketitle

\begin{abstract}
We consider sub-Riemannian manifolds which are homogeneous spaces equipped with a natural sub-Riemannian structure induced by a transitive action by a Lie group. In such a setting, the corresponding  sub-Laplacian is not an elliptic but a hypoelliptic operator. We study logarithmic Sobolev inequalities with respect to the hypoelliptic heat kernel measure on such homogeneous spaces. We show that the logarithmic Sobolev constant can be chosen to depend only on the Lie group acting transitively on such a homogeneous space but the constant is independent of the action of its isotropy group. This approach allows us to track the dependence of the logarithmic Sobolev constant on the geometry of the underlying space, in particular we are able to show that the logarithmic Sobolev constants is independent of the dimension of the underlying spaces in several examples. We illustrate the results by considering the Grushin plane, non-isotropic Heisenberg groups, Heisenberg-like groups, Hopf fibration, $\operatorname{SO}(3)$, $\operatorname{SO}(4)$, and compact Heisenberg manifolds.
\end{abstract}

\tableofcontents

\section{Introduction}

The logarithmic Sobolev inequality  has been first introduced and studied by L. Gross in \cite{Gross1975c} on a Euclidean space with respect to the Gaussian measure, and since then it found many applications. In particular, many existing results concern the question on how the constant in the logarithmic Sobolev inequality depends on the geometry of the underlying space, mostly in the Riemannian manifold setting, see for example \cite[Section 5.7]{BakryGentilLedouxBook}. The logarithmic Sobolev constant in that case  depends on the Ricci lower bound while it is independent of the dimension. This functional inequality is closely related to many important properties of the corresponding Markov semigroup such as hypercontractivity. Moreover, the fact that the logarithmic Sobolev constant often does not depend on the dimension makes it applicable in  infinite-dimensional settings.

Such results in the Riemannian setting rely on ellipticity of the Laplace-Beltrami operator as well as on geometric methods such as a curvature-dimension inequality, or different versions of  $\Gamma$ calculus. It is natural to consider logarithmic Sobolev inequalities on sub-Riemannian manifolds, which are curved spaces with more degeneracies than Riemannian manifolds. However, there are some fundamental difficulties.  The corresponding Laplacians are not elliptic operators but hypoelliptic which makes analysis more challenging. In addition, the Riemannian curvature-dimension condition is not available in general. While recently such geometric methods have been developed for some sub-Riemannian manifolds starting with \cite{BaudoinGarofalo2017}, they are not easily applicable to general sub-Riemannian manifolds. 

Logarithmic Sobolev inequalities in the sub-Riemannian setting have been studied for a number of examples such as isotropic and non-isotropic Heisenberg groups. There are different approaches to study the inequality, we refer only to the most relevant publications 
\cite{
BakryBaudoinBonnefontChafai2008, 
BaudoinBonnefont2012,
BonnefontChafaiHerry2020,
DagherZegarlinski2022a,
Eldredge2010, 
FrankLieb2012,
GordinaLuo2022, 
HebischZegarlinski2010, 
LiHong-Quan2006,  
Zhang-Y2021note}. 

In the current paper, we study logarithmic Sobolev inequalities on a larger class of sub-Riemannian manifolds than Lie groups. Namely, we consider homogeneous spaces with a natural sub-Riemannian structure inherited from that of their corresponding transitive acting Lie group. Note that such spaces include the class of sub-Riemannian homogeneous spaces considered in \cite{Strichartz1986}, and we comment on this terminology in more detail in Section~\ref{sec.SubRiemannianStructureonHomogoneousSpaces}. We then study logarithmic Sobolev inequalities with respect to a hypoelliptic heat kernel measure on such spaces. Also we discuss how the logarithmic Sobolev constant depends on the geometry of the underlying space. In Theorem~\ref{thm.LSIonHomogeneousSpace} we show that the logarithmic Sobolev constant only depends on the Lie group acting transitively on such a homogeneous space but it is independent of the action of its isotropy group. That is, if the same Lie group acts transitively on two homogeneous spaces, then these spaces might satisfy a logarithmic Sobolev inequality with the same constant. For some examples, this method allows us to show  that the logarithmic Sobolev constant does not depend on the dimension of the underlying spaces.

The motivation for our approach comes from the work of L.~Gross in the Riemannian setting. In \cite{Gross1992}, the connection between logarithmic Sobolev inequalities with respect to heat kernel measures on a connected Lie group and quotient spaces was studied. This is the starting point of our study in the sub-Riemannian setting. 

Our approach relies on the symmetry of the underlying space as expressed by the group action, which enables us to reduce the study of the logarithmic Sobolev inequality to that of a quotient space. This approach has been used for the heat kernel analysis on homogeneous spaces. For example, in \cite{DriverGrossSaloff-Coste2010}, B.~Driver, L.~Gross and L.~Saloff-Coste used it to prove that the Taylor map on complex manifolds is unitary. In \cite{Luo2023}, this approach can be applied to study Poincar\'e inequalities on homogeneous spaces. For homogeneous spaces which are Lie groups, tensorization and projection was used in \cite{BaudoinGordinaSarkar2023} to study other functional inequalities. 

Our paper is organized as follows. We first review basics of sub-Riemannian geometry and describe our setting in Section~\ref{sec.Preliminaries}. Then  in Section~\ref{sec.HomogeneousSpaces} we deduce a logarithmic Sobolev inequality with respect to the heat kernel measure on a homogeneous space equipped with a natural sub-Riemannian structure. Finally we present some examples which our main result applies to in Section~\ref{sec.Examples}, \ref{sec.ExamplesNilmanifold} and \ref{sec.ExampleSU(2)Type},  including the Grushin plane, non-isotropic Heisenberg groups, Heisenberg-like groups, Hopf fibration, $\operatorname{SO}(3)$, $\operatorname{SO}(4)$, compact Heisenberg manifolds.

\section{Preliminaries} \label{sec.Preliminaries}

In this paper, we consider some new examples of sub-Riemannian manifolds. They are connected homogeneous spaces that can be equipped with a natural sub-Riemannian structure inherited from a class of connected Lie groups. We first review some standard definitions in sub-Riemannian geometry, and then, we describe how a natural left-invariant sub-Riemannian structure on such a connected Lie group is constructed. After that, we will explain how this class of Lie groups can give rise to some new examples of sub-Riemannian manifolds, which are homogeneous spaces with a natural sub-Riemannian structure.

\subsection{Sub-Riemannian manifolds}

Let $M$ be an $n$-dimensional connected smooth manifold and we denote by $TM$ its tangent bundle. 
Suppose $\mathcal{H}$ is a smooth sub-bundle of $TM$.

\begin{definition} [H\"ormander's condition] 
If each fiber $\mathcal{H}_p$ of the sub-bundle $\mathcal{H}$ at every $p\in M$ has dimension $m$ for some $m\leqslant n$, then we call $\mathcal{H}$ a \emph{horizontal distribution} on $M$. We say that the distribution $\mathcal{H}$ satisfies \emph{H\"ormander's condition} if any local frame of $\mathcal{H}$ together with their finitely iterated Lie brackets span the tangent bundle $TM$. 
\end{definition}

\begin{definition}
Suppose $M$ together with its horizontal distribution $\mathcal{H}$ satisfies the H\"ormander's condition. If each fiber $\mathcal{H}_p$ for every $p\in M$ has dimension $m$ and is equipped with an inner product $\langle\cdot,\cdot\rangle_{\mathcal{H}}$ which varies smoothly between fibers, then the inner product $\langle\cdot,\cdot\rangle_{\mathcal{H}}$ is called a \emph{sub-Riemannian metric} and the triple $\left(M,\mathcal{H},\langle\cdot,\cdot\rangle_{\mathcal{H}}\right)$ is called a \emph{sub-Riemannian manifold of rank $m$}. 
\end{definition}
Note that when $m=n$, then $M$ is a Riemannian manifold.

\begin{definition} 
Sections of $\mathcal{H}$ are called \emph{horizontal vector fields}, and curves on $M$ whose velocity vectors are horizontal are called \emph{horizontal curves}. The \emph{length} of a horizontal curve $\gamma: [a, b] \longrightarrow M$ is defined to be
\[
l_{\mathcal{H}}\left( \gamma \right)=\int_{a}^{b} \sqrt{\langle\gamma^{\prime}(t),\gamma^{\prime}(t)\rangle_{\mathcal{H}}} dt.
\]
If $\gamma$ is not horizontal we define $l_{\mathcal{H}}\left( \gamma \right)=\infty$.

The \emph{Carnot-Carath\'{e}odory distance} between $p_{1}, p_{2} \in M$ is defined as

\begin{align}\label{df.CCdistance}
d_{CC}(p_{1}, p_{2}):=\inf \left\{ l_{\mathcal{H}}\left( \gamma \right): \gamma\left( a \right)=p_{1}, \gamma\left( b \right)=p_{2} \right\}.
\end{align}
\end{definition}
The Chow-Rashevsky theorem asserts  that H\"{o}rmander's condition implies that any two points in $M$ can be joined by a horizontal path, therefore  $d_{CC}(p_{1}, p_{2})$ is finite for any $p_{1}, p_{2} \in M$. For more details, we refer to \cite{AgrachevBarilariBoscainBook2020, MontgomeryBook2002} et al. Moreover, $M$ is a metric space with respect to the Carnot-Carath\'{e}odory distance (see \cite{Caratheodory1909}) and the topology on $M$ induced by $d_{CC}$ agrees with the original manifold topology of $M$ by \cite[Theorem 2.1.3]{MontgomeryBook2002}.

It is not clear whether $\left(M, d_{CC}\right)$ is a complete metric space in general. If a sub-Riemannian manifold $M$ is compact as a metric space, then it is complete since it is true for metric spaces. A characterization of $\left(M, d_{CC}\right)$ being complete is given in \cite[Theorem 7.3]{Strichartz1986}.

\begin{definition} \label{df.HorizontalGradient}
For any $f\in C^{\infty}(M)$, its \emph{horizontal gradient} $\nabla_{\mathcal{H}}f$ is a horizontal vector field such that for any $X\in\mathcal{H}$,
\begin{align*}
\langle \nabla_{\mathcal{H}}f,X\rangle_{\mathcal{H}}=Xf.
\end{align*}
\end{definition}

If for every $p\in M$, there is a neighborhood $U$ of $p$ and a collection of smooth vector fields $\{X_1,\cdots,X_m\}$ defined on $U$ such that they are orthonormal with respect to the sub-Riemannian metric $\langle\cdot,\cdot\rangle_{\mathcal{H}}$, then the horizontal gradient $\nabla_{\mathcal{H}}f$ has the form
\begin{align*}
\nabla_{\mathcal{H}}f=\sum_{i=1}^m(X_if)X_i.
\end{align*}

\begin{definition}
A second order differential operator $\Delta_{\mathcal{H}}$ defined on $C^{\infty}(M)$ is called a \emph{sub-Laplacian} if for every $p\in M$, there is a neighborhood $U$ of $p$ and a collection of smooth vector fields $\{X_0,X_1,\cdots,X_m\}$ defined on $U$ such that $\{X_1,\cdots,X_m\}$ are orthonormal with respect to the sub-Riemannian metric $\langle\cdot,\cdot\rangle_{\mathcal{H}}$ and
\begin{align*}
\Delta_{\mathcal{H}}:=\sum_{i=1}^m(X_i)^2+X_0.
\end{align*}
\end{definition}

By the classical result in \cite{Hormander1967a}, H\"{o}rmander's condition implies that any sub-Laplacian is hypoelliptic. Furthermore, any sub-Laplacian is a diffusion operator which is locally subelliptic, e.~g. \cite{JerisonSanchez-Calle1987, FeffermanSanchez-Calle1986}

Here we would like to comment on the choice of a reference measure in sub-Riemannian geometry. As there is no canonical reference measure such as the Riemannian volume for Riemannian manifolds, the analysis related to the sub-Laplacian might depend on such a choice. Standard choices of measures in sub-Riemannian geometry include Popp's measure (see \cite[Section 10.6]{MontgomeryBook2002} for details) and the Hausdorff measure (see \cite[Section 2.8]{MontgomeryBook2002} for details). We discuss this issue in Section~\ref{s.MeasuresHS} and in several concrete cases in Section~\ref{sec.Examples}. 
While it is an interesting problem to study how analytic properties of the heat semigroup with the sub-Laplacian as its infinitesimal generator depend on the reference measure, this is not the main focus of this paper.

\subsection{Dirichlet forms on sub-Riemannian manifolds} \label{sec.DirichletFormSubRiemannian}

Let $M$ be an $n$-dimensional sub-Riemannian manifold and $\mu$ a smooth non-vanishing measure on $M$ such that
\begin{align} \label{eqn.MeasureSubRiemannian}
d\mu= \rho dx_1\wedge \cdots \wedge dx_n
\end{align}
where $\rho \in C^{\infty}(M)$ with $\rho>0$ and $x_1,\cdots,x_n$ are local coordinates on $M$. Let 
\begin{align*}
\mathcal{E}^0_{\mu}(f, h):=\int_{M} \left\langle \nabla_{\mathcal{H}}f, \nabla_{\mathcal{H}}h \right\rangle_{\mathcal{H}} d\mu
\end{align*}
for any $f, h \in C^{\infty}_c(M)$. We denote $\mathcal{E}^0_{\mu}(f):=\mathcal{E}^0_{\mu}(f, f)$. 
The bilinear form  $\mathcal{E}^0_{\mu}$ can be extended to a Dirichlet form by 
\cite[p.346-347]{Varopoulos1988a}, which we include below for completeness.

\begin{theorem} [\cite{Varopoulos1988a}, pp.346-347]\label{thm.DirichletFormSubRiemannian}
Suppose the measure $\mu$ is given by \eqref{eqn.MeasureSubRiemannian}, then the bilinear form $\mathcal{E}^0_{\mu}$ is closable on $L^2(M, d\mu)$ with respect to the norm $\Vert \cdot\Vert_{\mathcal{E}_{\mu}}:=\Vert \cdot\Vert_{L^2(M, d\mu)}+\left(\mathcal{E}^0_{\mu}(\cdot)\right)^{\frac{1}{2}}$. Its closure $\mathcal{E}_{\mu}$ together with its domain $\mathcal{D}\left(\mathcal{E}_{\mu}\right) \subseteq L^2(M, d\mu)$ is a Dirichlet form on $L^2(M, d\mu)$.
\end{theorem}
From the construction $\mathcal{E}^0_{\mu}$ and $\mathcal{E}_{\mu}$, we see that $C^{\infty}_c(M)$ is dense in $\mathcal{D}\left(\mathcal{E}_{\mu}\right)$ under the norm $\Vert \cdot\Vert_{\mathcal{E}_{\mu}}$.

\subsection{Logarithmic Sobolev inequalities on sub-Riemannian manifolds} 

\begin{notation}\label{nota.LSI}
We say that $M$ satisfies a logarithmic Sobolev inequality with the constant $C\left(M,\mathcal{H},\mu\right)$ if
\begin{align} \label{LSI}
\int_{M}f^2\log f^2d\mu-\left(\int_{M}f^2 d\mu\right)\log\left(\int_{M}f^2d\mu\right) \leqslant C\left(M,\mathcal{H}, \mu\right) \mathcal{E}_{\mu}(f)
\end{align}
for any $f\in \mathcal{D}\left(\mathcal{E}_{\mu}\right)$. In such a case we also say that $LSI_C(M,\mathcal{H},\mathcal{D}\left(\mathcal{E}_{\mu}\right),\mu)$ holds. 
\end{notation}

To show that a logarithmic Sobolev inequality holds for all functions from $\mathcal{D}\left(\mathcal{E}_{\mu}\right)$, it is enough to show such an inequality for functions from  $C^{\infty}_c(M)$. This is because $C^{\infty}_c(M)$ is dense in $\mathcal{D}\left(\mathcal{E}_{\mu}\right)$ with respect to the norm $\Vert \cdot\Vert_{\mathcal{E}_{\mu}}$, which comes from  the closability of $\mathcal{E}^0_{\mu}$. For completeness, we include the proof here.

\begin{proposition} \label{prop.LimitInLSI}
If \eqref{LSI} holds for functions from $C^{\infty}_c(M)$, then $LSI_C\left(M,\right.
\\
\left.\mathcal{H},\mathcal{D}\left(\mathcal{E}_{\mu}\right),\mu\right)$ holds too.
\end{proposition}

\begin{proof}
The closability of $\mathcal{E}^0_{\mu}$ defined on $C^{\infty}_c(M)$ implies that $C^{\infty}_c(M)$ is dense in $\mathcal{D}\left(\mathcal{E}_{\mu}\right)$ under the norm $\Vert \cdot\Vert_{\mathcal{E}_{\mu}}$. In this way, for any $f\in\mathcal{D}\left(\mathcal{E}_{\mu}\right)$, there exists a sequence $\{f_k\}_{k=1}^{\infty
} \subseteq C^{\infty}_c(M)$ such that $\Vert f_k\Vert_{L^2(M, d\mu)}\to \Vert f\Vert_{L^2(M, d\mu)}$ and $\mathcal{E}_{\mu}(f_k)\to \mathcal{E}_{\mu}(f)$ as $k\to\infty$. Then we can take a subsequence (still denoted by $\{f_k\}_{k=1}^{\infty
}$) such that $f_k\to f$ for $\mu$-a.e. as $k\to\infty$. Taking the limit as $k\to\infty$ in the following logarithmic Sobolev inequality for $f_k$
\begin{align*} 
\int_{M}(f_k)^2\log (f_k)^2d\mu-\left(\int_{M}(f_k)^2 d\mu\right)\log\left(\int_{M}(f_k)^2d\mu\right) \leqslant C\left(M,\mathcal{H}, \mu\right) \mathcal{E}_{\mu}(f_k),
\notag
\end{align*}
we obtain the desired result since the form $\mathcal{E}_{\mu}$ on $L^2(M, d\mu)$ is closed.
\end{proof}

\subsection{Left-invariant sub-Riemannian structure on a Lie group $G$} \label{sec.LieGroupSubRiemannian}

Let $G$ be an $(n+m)$-dimensional connected Lie group. We identify its Lie algebra $\mathfrak{g}$ with the tangent space $T_eG$ at the identity $e$. 

Throughout this paper, we assume that H\"ormander's condition is satisfied on $G$. That is, there exist a family of linearly independent vectors $\{X_1,\cdots,X_n\}\subseteq \mathfrak{g}$ such that their Lie brackets span the whole Lie algebra $\mathfrak{g}$. We will call $\mathcal{H}:=\operatorname{Span}\{X_1,\cdots,X_n\}$ the \emph{horizontal space}. Under the assumption of H\"ormander's condition, $G$ can be equipped with a natural left-invariant sub-Riemannian structure, which we will explain in details.

We first recall some facts about left-invariant vector fields on $G$. Recall that the exponential map $\operatorname{exp}:\mathfrak{g}\rightarrow G$ is a local diffeomorphism, see e.~g. \cite[pp. 49-50]{KnappBook1996}. We can also identify $\mathfrak{g}$ with the collection of left-invariant vector fields as follows.

\begin{notation}
For any $X\in\mathfrak{g}$, we denote  by $\widetilde{X}$ the left-invariant vector field such that $\widetilde{X}(e)=X$, that is, for any $f\in C^{\infty}(G)$, we have
\begin{align*}
\left(\widetilde{X}f\right)(g):=\left.\frac{d}{dt}\right\vert_{t=0}f\left(g\operatorname{exp}(tX)\right).
\end{align*}
\end{notation}
Next, we describe the natural left-invariant sub-Riemannian structure on $G$. 

\begin{definition}
We say that the distribution $\mathcal{H}$ is \emph{left-invariant} if
$\mathcal{H}_e$ is a linear subspace of $\mathfrak{g}$ and, at any $g\in G$, $\mathcal{H}_g$ is the left translation of $\mathcal{H}_e$. We say that the sub-Riemannian metric $\langle\cdot,\cdot\rangle_{\mathcal{H}}^G$ is \emph{left-invariant} if at any $g\in G$, $\langle\cdot,\cdot\rangle_{\mathcal{H}_g}^G$ is the left translation of $\langle\cdot,\cdot\rangle_{\mathcal{H}_e}^G$. We say that $\left(G,\mathcal{H},\langle\cdot,\cdot\rangle_{\mathcal{H}}\right)$ is a \emph{left-invariant sub-Riemannian structure} if both $\mathcal{H}$ and $\langle\cdot,\cdot\rangle_{\mathcal{H}}^G$ are left-invariant.
\end{definition}

Suppose there is an inner product $\langle\cdot,\cdot\rangle_{\mathcal{H}}$ on the horizontal space $\mathcal{H}$. Then we use the left-translation to define the sub-bundle $\mathcal{H}^{G}$ with the induced left-invariant sub-Riemannian metric $\langle \cdot, \cdot \rangle^G_{\mathcal{H}}$. That is, for any $g\in G$, the \emph{horizontal distribution} is
\begin{align*}
\mathcal{H}_g=\{\widetilde{X}(g):X\in \mathcal{H}\}
\end{align*}
and the left-invariant inner product $\left\langle\cdot,\cdot\right\rangle^{G}_{\mathcal{H}}$ is chosen in such a way that
\begin{align*}
\langle \widetilde{X}(g),\widetilde{Y}(g)\rangle^G_{\mathcal{H}_g}:=\langle X,Y\rangle_{\mathcal{H}}
\end{align*}
for any $X,Y\in\mathcal{H}\subseteq \mathfrak{g}$. Similarly, we can define the norm $\vert \cdot \vert_{\mathcal{H}}$ induced by $\langle \cdot, \cdot \rangle_{\mathcal{H}}$ on $\mathcal{H}$ and then use the left translation to define the left-invariant norm denoted by $\vert\cdot\vert_{\mathcal{H}_g}$ on $\mathcal{H}^{G}_g$ for any $g\in G$. Thus, we have obtained a natural sub-Riemannian structure $\left(G,\mathcal{H}^{G},\langle\cdot,\cdot\rangle^{G}_{\mathcal{H}}\right)$ on $G$.

Moreover, if $\{X_1,\cdots,X_n\}$ forms an orthonormal basis for $\mathcal{H}$, then left-invariant vector fields $\{\widetilde{X_1},\cdots,\widetilde{X_n}\}$ induced by $\{X_1,\cdots,X_n\}$ will be an orthonormal frame for the sub-bundle $\mathcal{H}^{G}$. 

In this setting, for any $f\in C^{\infty}(G)$, the \emph{horizontal gradient} $\nabla_{\mathcal{H}}^{G}f$ has the form
\begin{align*}
\nabla_{\mathcal{H}}^{G}f= \sum_{i=1}^{n}
(\widetilde{X_i}f)\widetilde{X_i}.
\end{align*}

In addition, the Carnot-Carath\'eodory distance $d_{CC}^G$ is well-defined on $G$ and is a left-invariant metric on $G$, that is, for any $g_{1},g_{2},g\in G$
\begin{align*}
& d_{CC}^{G}(g_{1}, g_{2})=d_{CC}^{G}((g_{2})^{-1}g_{1}, e),
\\
& d_{CC}^{G}(g^{-1},e)=d_{CC}^{G}(g,e).
\end{align*}
Also, $\left(G, d_{CC}^{G}\right)$ is a complete metric space by \cite[pp. 682]{AgrachevBarilariBoscainBook2020}.

\subsection{Sub-Laplacian and hypoelliptic heat kernel measure on Lie groups}

H\"{o}rmander's condition implies that the \emph{sub-Laplacian}
\begin{equation}\label{e.SubLaplacian}
\Delta_{\mathcal{H}}^{G}:=\sum_{i=1}^{n} \left(\widetilde{X_{i}}\right)^2
\end{equation}
is a hypoelliptic operator by the classical result in \cite{Hormander1967a}. In particular, the sub-Laplacian only depends on the sub-Riemannian metric $\langle \cdot,\cdot\rangle_{\mathcal{H}}^G$ but it is independent of the choice of orthonormal frame by \cite[Theorem 3.8]{GordinaLaetsch2016a}.

Next, we define the hypoelliptic heat kernel measure on $G$. First we choose a right-invariant Haar measure $\mu_R^G$ on $G$. The sub-Laplacian $\Delta_{\mathcal{H}}^{G}$ is essentially self-adjoint on $C_{c}^{\infty}\left( G \right)$ in $L^{2}\left( G, d\mu_R^G \right)$ by \cite[pp. 950]{DriverGrossSaloff-Coste2009a}. The corresponding semigroup denoted by $P_t^G$  admits a probability transition kernel $\mu_t^{G}\left( g, dh \right)$ such that $\mu_t^{G}\left( g, A \right)\geqslant 0$ for all Borel sets $A$ and
\[
\left( P_t^Gf \right)\left( g \right)=\int_{G} f\left( h \right) \mu_t^{G}\left( g, dh \right)
\]
for all $f \in L^{2}\left( G, dg \right)$.

As explained in \cite[p. 952]{DriverGrossSaloff-Coste2009a} the transition kernel measure $\mu_t^{\omega}\left( g, dh \right)$ admits a continuous density, $p_t^{G}\left( g, h \right)$, with respect to the right Haar measure $\mu_R^G$ and

\begin{align}\label{eqn.HeatKernelMeasureDF}
\mu_t^{G}\left( g, dh \right)=p_{t}^{G}\left( g, h \right)d\mu_R^G(h).
\end{align}
Note that the sub-Laplacian $\Delta_{\mathcal{H}}^{G}$ commutes with left translations which together with the right invariance of the right Haar measure imply that

\begin{align} \label{eqn.LeftInvariance}
p_{t}^{G}\left( g, h \right)=p_{t}^{G}\left( e, g^{-1}h \right)m(g),
\end{align}
where $m$ is the modular function defined by 
\[
\int_Gf(hg)d\mu_R^G(g)=m(h)\int_Gf(g)d\mu_R^G(g),
\] 
therefore it suffices to look at the function $p_{t}^{G}\left( e, g\right)$. From now on we use $p_{t}^{G}\left( g \right)$ to denote this function and we will refer to it as the \emph{heat kernel}.

\begin{remark}
When $G$ is unimodular, that is, $m(h)=1$ for any $h\in G$, the right Haar measure $\mu_R^G$ and the left Haar measure $\mu_L^G$ coincide, that is, the reference measure $\mu_R^G=\mu_L^G$ is bi-invariant. In particular, nilpotent groups and groups of compact type are unimodular. 
\end{remark}

\begin{definition} \label{df.HeatKernelMeasure}
We call a family of measures $\{\mu_t^{G}\}_{t>0}$ on $G$ with
\[
d\mu_t^{G}\left(  g \right)=\mu_t^{G}\left(  dg \right)=\mu_t^{G}\left( e, dg \right)=p_{t}^{G}\left( g \right)d\mu_R^G(g)
\]
the \emph{heat kernel measure}.
\end{definition}

From Definition~\ref{df.HeatKernelMeasure}, it seems that the heat kernel measure depends on the choice of the reference measure on $G$, which in our case is the right Haar measure $\mu_R^G$ on $G$. However, there is an equivalent way to describe the heat kernel measure that does not need to fix your choice of the reference measure. Namely, we have
\begin{align*}
\mu_t^G(A)=P_t^G\mathbbm{1}_A
\end{align*}
for any $A\in\mathcal{B}(G)$.

\subsection{Dirichlet form associated to the heat kernel measure on Lie groups} \label{sec.DirichletFormonG}

We now consider a Dirichlet form on a Lie group $G$ with respect to the hypoelliptic heat kernel measure $\mu_t^G$ on $G$. In particular, we can apply the approach in  Section~\ref{sec.DirichletFormSubRiemannian} to this setting. 

We define
\begin{align}
\mathcal{E}^0_{\mu_t^G}(f, h): & =\int_{G} \left\langle \nabla_{\mathcal{H}}f,\nabla_{\mathcal{H}}h\right\rangle_{\mathcal{H}} d\mu_t^G
\\
&
=\int_{G} \left\langle \nabla_{\mathcal{H}}f,\nabla_{\mathcal{H}}h\right\rangle_{\mathcal{H}} p_t^G(g)d\mu_R^G(g) \notag
\end{align}
for any $f, h \in C^{\infty}_c(G)$.

By the existence, smoothness and positivity (see \cite[Theorem 3.4 (i)]{DriverGrossSaloff-Coste2009a}) of the heat kernel $p_t^G$ on $G$, we see that the heat kernel measure $\mu_t^G$ satisfies \eqref{eqn.MeasureSubRiemannian}. By Theorem~\ref{thm.DirichletFormSubRiemannian}, the bilinear form $\mathcal{E}^0_{\mu_t^G}$ is closable on $L^2(G, d\mu_t^G)$. The closure denoted by $\mathcal{E}_{\mu_t^G}$ of $\mathcal{E}^0_{\mu_t^G}$ together with its domain $\mathcal{D}\left(\mathcal{E}_{\mu_t^G}\right) \subseteq L^2(G, d\mu_t^G)$ is a regular Dirichlet form on $L^2(G, d\mu_t^G)$.  We can choose $C^{\infty}_c(M)$ as the core of $\mathcal{E}_{\mu_t^G}$.

There is a precise description of the domain $\mathcal{D}\left(\mathcal{E}_{\mu_t^G}\right) \subseteq L^2(G, d\mu_t^G)$ of this Dirichlet form $\mathcal{E}_{\mu_t^G}$ in \cite[p.1886]{Lust-Piquard2010}. We include it here for completness. 

\begin{lemma} [p.~1886 in \cite{Lust-Piquard2010}]
The domain $\mathcal{D}\left(\mathcal{E}_{\mu_t^G}\right)$ of $\mathcal{E}_{\mu_t^G}$ has the form
\begin{align} \label{eqn.DomainWeightedDirichletForm}
\mathcal{D}\left(\mathcal{E}_{\mu_t^G}\right)=\left\{f\in L^2\left(G, d\mu_t^{G}\right):X_if\in L^2\left(G, d\mu_t^{G}\right), i=1,\cdots, n\right\}.
\end{align}
\end{lemma}

This Dirichlet form $\mathcal{E}_{\mu_t^G}$ is closely related to the Ornstein–Uhlenbeck process and the Ornstein–Uhlenbeck semigroup on $G$, for more details we refer to \cite{BaudoinHairerTeichmann2008, Lust-Piquard2010}.

\section{Logarithmic Sobolev inequalities on homogeneous spaces} \label{sec.HomogeneousSpaces}

\subsection{Main result}

In this section, we prove a logarithmic Sobolev inequality on the homogeneous space equipped with a natural sub-Riemannian structure induced by a transitive action by a Lie group. We start with basic definitions and present our main result, Theorem~\ref{thm.MainTheorem}.

\begin{definition}
A smooth manifold endowed with a transitive smooth action by a Lie group G is called a \emph{homogeneous
G-space} or a \emph{homogeneous space}. The \emph{isotropy group} of $p \in  M$ is defined as
$G_{p} := \left\{ g \in G: g \cdot p = p \right\}$. 
\end{definition}

\begin{theorem} \label{thm.MainTheorem}
Let $M$ be a homogeneous space with a connected Lie group $G$ acting transitively on it. Suppose that $G$ is equipped with a sub-Riemannian structure $\left(G,\mathcal{H}^{G},\langle\cdot,\cdot\rangle^{G}_{\mathcal{H}}\right)$ and the logarithmic Sobolev inequality \eqref{LSI} holds for  $f\in C^{\infty}_c(G)$ with the constant $C\left(G, \mathcal{H}^G, \mu_t^G \right)$. Then 
\begin{enumerate}
    \item There is a natural sub-Riemannian structure $\left(M,\mathcal{H}^{M}, \langle\cdot, \cdot\rangle^{M}_{\mathcal{H}}\right)$ on $M$ induced by the transitive action by $G$.
    \item There exists the hypoelliptic \emph{heat kernel measure} $\mu_t^M$ on $M$ such that the  heat equation holds.
    \item There exists a Dirichlet form $\mathcal{E}_{M}$ associated with the hypoelliptic heat kernel measure $\mu_t^M$ with the domain $\mathcal{D}\left(\mathcal{E}_{M}\right) \subseteq L^2\left(M,d\mu_t^M\right)$.
    \item The hypoelliptic logarithmic Sobolev inequality $LSI_C\left( M, \mathcal{H}^{M}, \mu_t^{M}\right)$ holds. Moreover, the constant $C\left(M, \mathcal{H}^{M}, \mu_t^{M}\right)$ can be chosen to be
    \begin{align*}
C\left(M,\mathcal{H}^{M},\mu_t^{M}\right)=C\left(G, \mathcal{H}^G,\mu^G_t\right).
    \end{align*}
\end{enumerate}
\end{theorem}
The key ingredient in the proof of this theorem is the characterization of homogeneous spaces given in Theorem~\ref{t.HomSpaceCharacterization}, which allows us to describe a homogeneous space as a quotient space and transform the study of our problem on a quotient space. In this way, we can describe the sub-Riemannian structure, reference measures, heat semigroup, hypoelliptic heat kernel measure and Dirichlet form associated to the hypoelliptic heat kernel measure on the homogeneous space precisely. In particular, we study  the heat kernel analysis on the homogeneous space $M$ in terms of the  group $G$ acting transitively $M$. This is the subject of several sections of the paper.

\subsection{Characterization of homogeneous spaces}

The fact that the action of $G$ on $M$ is transitive means that this structure \emph{looks the same} everywhere on the manifold. Any homogeneous $G$-space can be characterized by the following theorem.

\begin{theorem}[Homogeneous space characterization theorem, Theorem 21.~18 in \cite{LeeBook2003SmoothManifold}]\label{t.HomSpaceCharacterization}
Let $G$ be a Lie group, let $M$ be a homogeneous $G$-space, and let $p$ be any point of $M$. Then the isotropy group $G_p$ is a closed subgroup of $G$, and the map $F:G_p\backslash G \rightarrow M$ defined by $F(gG_p)=g\cdot p$ is an equivariant diffeomorphism.
\end{theorem} 
Note that all isotropy groups are conjugate. For this reason, even for different $p\in M$, the corresponding quotient manifold $G_p\backslash G$ is \emph{unique} up to a diffeomorphism. In this way, when we are given a smooth manifold $M$ together with a transitive action by a
Lie group $G$, it suffices to fix a point $p\in M$ and its corresponding isotropy group $G_p$ denoted by $H$. Then we can always use the preceding theorem to identify $M$ equivariantly
with a coset space of the form $H\backslash G$, where $H$ is a closed subgroup of $G$. Such a procedure enables us to use all of the machinery that is available  for analyzing quotient spaces. For the rest of this paper, we will identify such a homogeneous $G$-space $M$ with $H\backslash G$ and focus on the study of the quotient space $H\backslash G$.

Let $G$ be a connected Lie group and $H$ be a closed subgroup of $G$. Recall that by \cite[Theorem 20.12]{LeeBook2003SmoothManifold} $H$ is an embedded submanifold. Note that if $H$ is a discrete subgroup of $G$, then  $H$ is a closed Lie subgroup of dimension zero by \cite[Proposition 21.28]{LeeBook2003SmoothManifold}. Without loss of generality, we assume $H$ is an $k$-dimensional closed Lie subgroup of $H$ for some $k\geqslant 0$. We choose the right Haar measure $\mu^H_R$ on $H$ as the reference measure.  By \cite[Example 7.22 (c)]{LeeBook2003SmoothManifold}, the left action of $H$ on $G$
\begin{align}
& H \times G \rightarrow G \notag
\\
&
\left(h,g\right) \mapsto h \circ g \notag
\end{align}
is always smooth and free, but generally not transitive. When $H$ is discrete, this action is always proper by \cite[p. 557]{LeeBook2003SmoothManifold}. It is also interesting to see whether such an action is isometric in the sub-Riemannian sense. That is, we want to see whether under the action $H$ an orthonormal frame is mapped to an orthonormal frame of $\left(G, \mathcal{H}^{G}, \langle \cdot, \cdot \rangle^{G}_{\mathcal{H}}\right)$. We will discuss such an action for different examples in Section~\ref{sec.Examples}.

The right cosets $H\backslash G$ of $H$ have an induced smooth structure and form a $(n+m-k)$-dimensional smooth manifold. Let $\pi:G \rightarrow H\backslash G$ be the quotient map. We know that the continuity of the quotient map $\pi$ implies the connectedness of $ H\backslash G$ when $G$ is connected. Moreover, when $H$ is discrete, by \cite[Theorem~21.29]{LeeBook2003SmoothManifold}, $H\backslash G$ is a smooth manifold and the quotient map $\pi:G \rightarrow H\backslash G$ is a smooth normal covering map. However, $H\backslash G$ is not always a group. Only when $H$ is a normal subgroup, or equivalently when
 $\mathfrak{h}$, the Lie algebra of $H$ is an ideal in $\mathfrak{g}$, $H\backslash G$ is a Lie group.

\begin{notation}
We denote the homogeneous space $H\backslash G$ by $M$.
\end{notation}

We refer to \cite{DriverGrossSaloff-Coste2010} for a detailed description of the homogeneous space $M$ when $G$ is a simply connected Lie group and $H$ is a connected closed Lie subgroup. Note that connectedness of  $H$ is equivalent to $M$ being simply connected by \cite[ I. Chap. 1, Theorem 4.8]{GorbatsevichOnishchikVinbergBook1997}. However, we consider a more general situation here dropping the assumption that $H$ is a connected Lie subgroup. A typical case is when $H$ is a discrete subgroup. Then the homogeneous space $M$ is no longer simply connected as we will see in Section~\ref{sec.ExampleSO(3)} and Section~\ref{sec.ExampleSO(4)}.

\subsection{Smooth structure on a homogeneous space}
We first recall how one can define smooth vector fields on $M=H\backslash G$, as found in \cite[ Notation 6.3]{DriverGrossSaloff-Coste2010}.

\begin{definition}
For any $X\in\mathfrak{g}$ and any $m=Hg\in M$, we define 
\begin{align} \label{eqn.VectorFieldonHomogeneousSpace}
\dot{X}(m):=\left.\frac{d}{dt}\right\vert_{t=0}\left(m\operatorname{exp}\left(tX\right)\right)=\left.\frac{d}{dt}\right\vert_{t=0}H\left(g\operatorname{exp}\left(tX\right)\right).
\end{align}
\end{definition}
Then \eqref{eqn.VectorFieldonHomogeneousSpace} defines a smooth vector field on $M$. However, there is another way to understand smooth vector fields on $M$, which will be used more often in this paper. 

Note that both $\pi:G \rightarrow M$ and its differential $d\pi_g:T_gG\rightarrow T_{\pi(g)}M$ are surjective maps. We can connect smooth vector fields on $M$ with left-invariant vector fields on $G$ by the differential $d\pi_g$ of $\pi$. This link explicitly was given in \cite[Lemma 6.4]{DriverGrossSaloff-Coste2010}, when $G$ is simply connected and $H$ is connected. Lemma~\ref{l.3.5} described smooth vector fields on $M$ when $H$ is only a closed subgroup. Its proof is similar to that of \cite[Lemma 6.4]{DriverGrossSaloff-Coste2010}. For completeness, we include it below.

\begin{lemma}\label{l.3.5}
For any $X\in\mathfrak{g}$ and any $g\in G$, we have
\begin{align} \label{eqn.VectorRelation}
\left(d\pi_g(\widetilde{X})\right)(\pi(g))=\dot{X} \left(\pi(g)\right).
\end{align}
Moreover, for $X\in\mathfrak{g}$, $\dot{X}(He)=0$ if and only if $X\in\mathfrak{h}$.
\end{lemma}

\begin{proof}
For any $X\in\mathfrak{g}$ and any $g\in G$, we have
\begin{align*}
\left(d\pi_g(\widetilde{X})\right)(\pi(g)) & =\left.\frac{d}{dt}\right\vert_{t=0} \pi\left(g\operatorname{exp}(tX)\right)
\\
&
=\left.\frac{d}{dt}\right\vert_{t=0} \pi\left(g\right)\operatorname{exp}(tX)
\\
&
=\dot{X} \left(\pi(g)\right).
\end{align*}
By \eqref{eqn.VectorRelation}, we have $0=\dot{X} \left(He\right)=\left(d\pi_e(\widetilde{X})\right)(He)=\left(d\pi_e(X)\right)(He)$ if and only if $X\in\mathfrak{h}$. 
\end{proof}

\subsection{Sub-Riemannian structure on homogeneous spaces} \label{sec.SubRiemannianStructureonHomogoneousSpaces}
We start by introducing a natural sub-Riemannian structure on homogeneous spaces induced by the transitive action by $G$. Proposition~\ref{prop.SubRiemannianOnHomogeneousSpaces} proves part (1) of Theorem~\ref{thm.MainTheorem}.

\begin{proposition} \label{prop.SubRiemannianOnHomogeneousSpaces}
Suppose $G$ is equipped with a left-invariant sub-Riemannian structure $\left(G, \mathcal{H}, \langle\cdot,\cdot\rangle_{\mathcal{H}}\right)$, then the homogeneous space $M$ has a natural sub-Riemannian structure $\left(M,\mathcal{H}^{M},\langle\cdot,\cdot\rangle^{M}_{\mathcal{H}}\right)$ induced by the transitive action by $G$.
\end{proposition}

\begin{proof}
It suffices to show that H\"ormander's condition is satisfied on $M$. For any $\widetilde{X},\widetilde{Y}\in\mathcal{H}^G$ and any $g\in G$, we have
\begin{align} \label{eqn.HormanderonQuotientSpace}
\left[\left(d\pi_{g}(\widetilde{X})\right)(\pi(g)),\left(d\pi_{g}(\widetilde{Y})\right)(\pi(g))\right]=\left(d\pi_{g}\left(\left[\widetilde{X},\widetilde{Y}\right]\right)\right)(\pi(g)).
\end{align} 
This implies that H\"{o}rmander's condition is satisfied on $M$. 

Now, we describe how $M$ can be equipped with a natural sub-Riemannian structure. We choose the horizontal distribution $\mathcal{H}^{M}$ to be
\[
\mathcal{H}_{\pi(g)}^{M}= \operatorname{Span}\left\{\left(d\pi_{g}(\widetilde{X})\right)\left(\pi( g)\right): \widetilde{X}\in \mathcal{H}\, \text{such that} \, \left(d\pi_{g}(\widetilde{X})\right)\left(\pi(g)\right)\neq 0\right\}
\]  
for any $g\in G$. Note that $M$ may not always be constant rank. Then the sub-Riemannian metric $\langle\cdot,\cdot\rangle^{M}_{\mathcal{H}}$ on $\mathcal{H}^{M}$ may not simply be the push-forward metric of $\langle\cdot,\cdot\rangle^{G}_{\mathcal{H}}$ from $\mathcal{H}^G$ to $\mathcal{H}^{M}$. But it can be determined by the choice of an orthonormal frame induced from an orthonormal frame of $\left(G,\mathcal{H}^G,\langle\cdot,\cdot\rangle_{\mathcal{H}}^{G}\right)$. Let $\left\{\widetilde{X_1},\cdots,\widetilde{X_n}\right\}$ be an orthonormal frame of $\left(G,\mathcal{H}^G,\langle\cdot,\cdot\rangle_{\mathcal{H}}^{G}\right)$. At each $\pi(g)$ for any $g\in G$, $\langle\cdot,\cdot\rangle^{M}_{\mathcal{H}_{\pi(g)}}$ can be chosen in such a way that $\left\{\right.d\pi_g(\widetilde{X_i}):i=1,\cdots,n \, \text{such that} \left.\left(d\pi_{g}(\widetilde{X})\right)\left(\pi(g)\right)\text{are linearly independent}\right\}$ forms an orthonormal frame for it. Therefore $M$ has a natural sub-Riemannian structure $\left(M,\mathcal{H}^{M},\langle\cdot,\cdot\rangle^{M}_{\mathcal{H}}\right)$.
\end{proof}

\begin{remark}
The induced sub-Riemannian metric $\langle\cdot,\cdot\rangle^{M}_{\mathcal{H}}$ on $M$ is not always $G$-invariant. For example, when $G$ is the three-dimensional isotropic Heisenberg group and $M$ is the Grushin plane as we describe in Section~\ref{sec.Grushin}, the sub-Riemannian metric $\langle\cdot,\cdot\rangle^{M}_{\mathcal{H}}$ is not $G$-invariant.
\end{remark}

Note that the notion of \emph{sub-Riemannian homogeneous spaces} used in \cite[p.250]{Strichartz1986} is different from the setting considered in this paper. A sub-Riemannian homogeneous space is an example of homogeneous spaces equipped with natural sub-Riemannian structures induced by a transitive action by a Lie group. The transitive action of $G$ on a \emph{sub-Riemannian homogeneous space} is assumed to be an \emph{infinitesimal isometry} as defined in \cite[Section 8]{Strichartz1986}, while we do not require the action of $G$ to be such in our setting.

The Carnot-Carath\'{e}odory distance $d_{CC}^{M}$ induces a metric space structure on $M$ with the topology equivalent to the manifold topology. Generally, it is not clear under which assumptions $\left(M, d_{CC}^{M}\right)$ is a complete metric space. For example, the Grushin plane  considered in Section~\ref{sec.Grushin} is not complete or geodesically complete (see \cite[p.2]{PrandiRizziSeri2018} or \cite[p.12]{GalloneMichelangeliPozzoli2019}). Characterization of completeness of $\left(M, d_{CC}^{M}\right)$ is given in \cite[Proposition 1.1]{HughenThesis1995}. Moreover, when $M$ is also a Lie group, then $\left(M, d_{CC}^{M}\right)$ is complete. When $M$ is sub-Riemannian homogeneous space as defined in \cite{Strichartz1986}, then $\left(M, d_{CC}^{M}\right)$ is complete by \cite[p.250]{Strichartz1986}.

The \emph{horizontal gradient} $\nabla_{\mathcal{H}}^{M}f$ on $M$ is defined by
\begin{align*}
\left(\nabla_{\mathcal{H}}^{M}f\right)(\pi(g))=\left(d\pi_g\left(\nabla_{\mathcal{H}}^{G}(f\circ \pi)\right)\right)(\pi(g))
\end{align*}
for any $g\in G$. Then $\nabla_{\mathcal{H}}^{M}f$ on $M$ has the form 
\begin{align*}
\nabla_{\mathcal{H}}^{M}f= \sum_{i=1}^{n}
\left(\left(d\pi_{g}(\widetilde{X_i})\right)f\right)\left(d\pi_{g}(\widetilde{X_i})\right)
\end{align*}
for any $f\in C^{\infty}(M)$. In addition, we have
\begin{align} \label{eqn.GradientRelation}
\left\vert \nabla_{\mathcal{H}}^{G}(f\circ \pi)\right\vert_{\mathcal{H}}=\left\vert \nabla^{M}_{\mathcal{H}}f\right\vert_{\mathcal{H}} \circ \pi  
\end{align}
for any $f\in C^{\infty}\left(M\right)$. Finally we define the \emph{sub-Laplacian} on $M$ by
\begin{align} \label{df.SubLaplacianonQuotientSpace}
\Delta_{\mathcal{H}}^{M}=\sum_{i=1}^{n} \left(d\pi_{g}\left(\widetilde{X_i}\right)\right)^2,
\end{align}
and observe that $\Delta_{\mathcal{H}}^{M}$ is a hypoelliptic operator. Furthermore, we have
\begin{align} \label{eqn.LaplacianRelation}
\Delta^{G}_{\mathcal{H}}\left(f \circ \pi\right)=\left(\Delta^{M}_{\mathcal{H}} f\right) \circ \pi
\end{align}
for any $f\in C^{\infty}\left(M\right)$.

\subsection{Measures on homogeneous spaces}\label{s.MeasuresHS}

We would like to discuss two natural choices of a reference measure on $M$, namely, the pushforward measure of the right Haar measure $\mu_R^G$ to $M$ by the quotient map $\pi$ and a measure quasi-invariant with respect to the action by $G$ which satisfies the disintegration Theorem~\ref{thm.Disintegration}. These measures exist for any locally compact group $G$ together with a closed subgroup $H$, and in particular our setting when $G$ is a connected Lie group.

\subsubsection{Pushforward of the right Haar measure}

We first discuss the pushforward measure of the right Haar measure $\mu_R^G$ on $G$ to $M$ by the quotient map $\pi$. Recall that  $\mu_R^G$ is a Radon measure, and $\pi$ is a measurable map. Denote by $\pi_{\sharp}\mu_R^G$ the pushforward measure of $\mu_R^G$ to $H\backslash G$ by $\pi$

\begin{align*}
\left(\pi_{\sharp}\mu_R^G\right)(A):=\mu_R^G\left(\pi^{-1}(A)\right)
\end{align*}
for any $B\in \mathcal{B}(M)$.

\begin{lemma}
For any $A\in \mathcal{B}(M)$ and any $g\in G$, we have
\begin{align} \label{eqn.ImageRelation}
\pi^{-1}(Ag)=\left(\pi^{-1}(A)\right) g.
\end{align}
\end{lemma}

\begin{proof}
It suffices to show 
\begin{align} \label{eqn.PointwiseRelation}
\pi^{-1}(mg)=\pi^{-1}(m)g
\end{align}
for any $m\in M$ and any $g\in G$. Note that both sides of \eqref{eqn.PointwiseRelation} are Borel sets in $G$. For each $m\in M$, let us pick a representative $g_0\in G$ for the right coset of $H$ corresponding to $m$. 

For any $\tilde{g} \in \pi^{-1}(mg)$, we have $\pi(\tilde{g})=mg=(Hg_0)g=H(g_0g)$. This implies that there exists an $\tilde{h}\in H$ such that $\tilde{g}=\tilde{h}g_0g$. Then we have $\pi(\tilde{g}g^{-1})=\pi(\tilde{h}g_0gg^{-1})=\pi(\tilde{h}g_0)=Hg_0=m$, which implies $\tilde{g}\in \pi^{-1}(m)g$. Thus, $\pi^{-1}(mg)\subseteq \pi^{-1}(m)g$.

Now for any $\tilde{g}\in \pi^{-1}(m)g$, we have $\tilde{g}g^{-1}\in \pi^{-1}(m)$ and thus $\pi(\tilde{g}g^{-1})=m$. This means that there exists an $h^{\prime} \in H$ such that $\tilde{g}g^{-1}=h^{\prime}g_0$, so $\tilde{g}=h^{\prime}g_0g$. We see that $\pi(\tilde{g})=H(g_0g)=(Hg_0)g=mg$, so $\tilde{g}\in \pi^{-1}(mg)$. Thus, $\pi^{-1}(m)g \subseteq \pi^{-1}(mg)$. Therefore, \eqref{eqn.PointwiseRelation} holds and \eqref{eqn.ImageRelation} follows.
\end{proof}

\begin{theorem}
The pushforward measure $\pi_{\sharp}\mu_R^G$ on $M$ is $G$-invariant. That is, for any $A\in \mathcal{B}(M)$ and any $g\in G$, we have
\begin{align}
(\pi_{\sharp}\mu_R^G)(Ag)=(\pi_{\sharp}\mu_R^G)(A).
\end{align}
\end{theorem}

\begin{proof}
For any $A\in \mathcal{B}(M)$ and any $g\in G$, we see that
\begin{align*}
(\pi_{\sharp}\mu_R^G)(Ag) & =\mu_R^G\left(\pi^{-1}(Ag)\right)
\\
&
=\mu_R^G\left(\left(\pi^{-1}(A)\right) g\right)
\\
&
=\mu_R^G\left(\pi^{-1}(A)\right)=(\pi_{\sharp}\mu_R^G)(A),
\end{align*}
where the second equality is by \eqref{eqn.ImageRelation}. The third equality follows from the right invariance of $\mu_R^G$. Thus, $\pi_{\sharp}\mu_R^G$ is $G$-invariant.
\end{proof}

Moreover, we see that $\Delta^{M}_{\mathcal{H}}$ is symmetric on $L^2\left(M, d\left(\pi_{\sharp}\mu_R^G\right)\right)$.  That is, for any $f,f^{\prime}\in C^{\infty}_c(M)$, we have
\begin{align*}
\int_Mf\Delta^{M}_{\mathcal{H}}f^{\prime}d\left(\pi_{\sharp}\mu_R^G\right) & =\int_G(f\circ \pi) \Delta^{G}_{\mathcal{H}}(f^{\prime}\circ \pi)d\mu_R^G
\\
&
=\int_G \Delta^{G}_{\mathcal{H}}(f\circ \pi)(f^{\prime}\circ \pi)d\mu_R^G
\\
&
=\int_M f^{\prime} \Delta^{M}_{\mathcal{H}}fd\left(\pi_{\sharp}\mu_R^G\right).
\end{align*}
The first and fourth equality follow from \eqref{eqn.LaplacianRelation} and the change of variable formula for pushforward measures \cite[Theorem 1, Chapter V \S6]{BourbakiBookIntegrationI2004} under the map $\pi(g)\longmapsto m\in M$. The second equality follows from  the fact that $\Delta^{G}_{\mathcal{H}}$ is symmetric on $L^2\left(G, d\mu_R^G\right)$. Clearly the quotient map $\pi$ induces an isometry on the corresponding $L^{2}$ spaces as follows. For any $f\in L^2\left(M, d\left(\pi_{\sharp}\mu^G_R\right)\right)$, we have

\begin{align} \label{eqn.L^2NormReferenceMeasureRelation}
\Vert f\Vert_{L^2\left(M, d\left(\pi_{\sharp}\mu^G_R\right)\right)}=\Vert f\circ \pi\Vert_{L^2\left(G, d\mu^{G}_R\right)}.
\end{align}

In Section~\ref{sec.HeatSemigrouponHomogeneousSpaces}, we choose $\pi_{\sharp}\mu^G_R$ as our reference measure on $M$ when we study the heat semigroup on $M$.

\subsubsection{Quasi-invariant measures}

We now describe another choice of a reference measure on $M$. By \cite[Theorem 2 in Chapter VII \S2]{BourbakiBookIntegrationII2004}, there exists a measure $\mu^M$ on $M$, not necessarily unique as we will see in Theorem~\ref{thm.Disintegration}, which is quasi-invariant under the action of $G$.

An important property of the measure $\mu^M$ is that it satisfies a \emph{disintegration theorem} as below, which we include for completeness. In the case when $H$ is a maximal torus of a connected compact Lie group $G$, this becomes the classical Weyl Integration Formula. 
\begin{theorem}[Theorem 2 and (7) of Chapter VII
\S2 in \cite{BourbakiBookIntegrationII2004}] \label{thm.Disintegration}

Let $G$ be a locally compact group and $H$ a closed subgroup. Let $q:M \rightarrow G$ be the cross section map. Then there exists a continuous function $\rho:G\rightarrow \mathbb{R}$ such that
\begin{align*}
\rho(hg)=\frac{m_H(h)}{m_G(h)}\rho(g)   
\end{align*}
for any $g\in G$ and any $h\in H$.
Moreover, there exists a quasi-invariant measure $\mu^M$ under the action of $G$ such that for any $f\in C(G)$ we have
\begin{align*}
\int_Gf(g)\rho(g)d\mu_R^G(g)=\int_M\int_Hf\left(h\circ_G q(m)\right)d\mu_R^H(h)d\mu^M(m).
\end{align*}
\end{theorem}

\begin{remark}
Note that the existence and explicit form of $\mu^M$ both depend on the function $\rho$ which is not unique, so $\mu^M$ is not unique. However, all such measures $\mu^M$ have the same null sets by \cite[Theorem 1 of Chapter VII \S2 ]{BourbakiBookIntegrationII2004}. 
\end{remark}

\begin{remark}
Note that \cite{BourbakiBookIntegrationII2004} uses a left Haar measure on $G$ instead of $\mu_R^G$.  
\end{remark}
When the modular function of $G$ and that of $H$ coincide on $H$, $\mu^M$ is $G$-invariant by \cite[p. 462]{DriverGrossSaloff-Coste2009a} or \cite[Corollary 2, Chapter VII \S2]{BourbakiBookIntegrationII2004}. In particular, when $G$ is unimodular, by \cite[Proposition 10, Chapter VII \S2]{BourbakiBookIntegrationII2004} any closed subgroup $H$ is unimodular too. This is the case which applies to the examples in our paper. In \cite{DriverGrossSaloff-Coste2010}, the $G$-invariant measure $\mu^M$ is chosen as the reference measure on $M$. The disintegration theorem above has a simpler form as below and we refer to \cite[Proposition 4 (b), Chapter VII \S1]{BourbakiBookIntegrationII2004} or in \cite[Theorem 2.51]{FollandHABook} for more details.

\begin{proposition}[Disintegration Theorem] \label{prop.Disintegration}
Let $G$ be a locally compact group and $H$ is a closed subgroup such that the modular function of $G$ and that of $H$ coincide on $H$. Then there exists a $G$-invariant measure $\mu^M$ such that for any $f\in C(G)$, we have
\begin{align*}
\int_Gf(g)d\mu_R^G(g)=\int_M\int_Hf\left(h\circ_G q(m)\right)d\mu_R^H(h)d\mu^M(m).
\end{align*}
\end{proposition}

\subsubsection{Connection between the measures on homogeneous spaces}

Now we would like to comment on the connection between the pushforward measure $\pi_{\sharp}(\mu_R^G)$ and a quasi-invariant measure $\mu^M$.

Generally, from \cite[Remark (1), Section 3, Chapter VII \S2]{BourbakiBookIntegrationII2004}, $\mu^M$ is a pseudo-image of $\rho \mu_R^G$ by $\pi$ as defined in \cite[Definition 1, Section 2, Chapter VI \S3]{BourbakiBookIntegrationI2004}.  When $H$ is compact, we can normalize $\mu_R^H$ and then $\mu^M$  coincides with $\pi_{\sharp}(\rho\mu_R^G)$ by \cite[Remark (2), Section 3, Chapter VII \S2]{BourbakiBookIntegrationII2004}.  From the above description, we see that $\pi_{\sharp}\mu_R^G$ and $\mu^M$ do not always coincide.

\subsection{Functions on homogeneous spaces}

There is a one-to-one correspondence between continuous functions with compact support on $G$ and $M$ as follows. For any $f\in C(G)$, we define 
\begin{align} \label{eqn.FunctionGtoH}
f_H(m):=\int_{H}f\left(h\circ_G q(m)\right)d\mu_R^H.
\end{align}
where $q: M \rightarrow G$ is the cross section map. We see that $f_H$ is a well-defined continuous function on $M$ and it is invariant under the left action of $H$. By \cite[p. 460]{DriverGrossSaloff-Coste2010}, $f_H\in C_c\left(M\right)$ if $f\in C_c(G)$. Moreover, the image of $C_c(G)$ under the map $f\mapsto f_H$ is $C_c\left(M\right)$, that is,  the map is surjective.

As for the opposite direction, \cite[Proposition 2.50]{FollandHABook} states that for any $f\in C_c(M)$, there exists $\widetilde{f}\in C_c(G)$ such that $\widetilde{f}_H=f$ and $q\left(\operatorname{supp}(f)\right)=\operatorname{supp}\left(\widetilde{f}\right)$.

There is a subtlety we would like to mention. Observe that given any $f\in C(M)$, we have $f\circ \pi\in C(G)$. However,
given any $f\in C_c(M)$, it is not always that $f\circ \pi\in C_c(G)$ and we only have that $f\circ \pi\in C_b(G)$.

\subsection{Heat semigroup and hypoelliptic heat kernel measure on homogeneous spaces} \label{sec.HeatSemigrouponHomogeneousSpaces}
We start with a discussion on the heat semigroup on $M$. In \cite{Hunt1956a}, Hunt discussed semigroups on homogeneous spaces. In addition to a number of other assumptions, in his setting $M$ is equipped with a $G$-invariant Riemannian structure. However, we look at a more general setting here.

\begin{definition}
We define the \emph{heat semigroup} $\{P_t^M\}_{t\geqslant 0}$ on $M$ by
\begin{align} \label{eqn.SemigroupRelation}
\left(P_t^Mf\right)\circ \pi=P_t^G(f\circ \pi) 
\end{align}
for any $f\in \mathcal{B}_b\left(M\right) \cap L^2\left(M, d\left(\pi_{\sharp}\mu^G_R\right)\right)$.
\end{definition}
Note that $\{P_t^M\}_{t\geqslant 0}$ is a family of linear operators on $L^2\left(M, d\left(\pi_{\sharp}\mu^G_R\right)\right)$.

\begin{lemma}
The family of linear operators $\{P_t^M\}_{t\geqslant 0}$ is a strongly continuous positivity preserving Markov semigroup on $L^2\left(M, d\left(\pi_{\sharp}\mu^G_R\right)\right)$.
\end{lemma}

\begin{proof}
We know that $\{P_t^G\}_{t\geqslant 0}$ is a strongly continuous, positivity preserving Markov semigroup on $L^2\left(G, d\mu^G_R\right)$. Relying on this fact and using the definition of $\{P_t^M\}_{t\geqslant 0}$ together with \eqref{eqn.L^2NormReferenceMeasureRelation}, we just need to apply the change of variable formula for pushforward measures \cite[Theorem 1, Chapter V \S6]{BourbakiBookIntegrationI2004} to the quotient map $\pi(g)\longmapsto m\in M$ to check that $\{P_t^M\}_{t\geqslant 0}$ also satisfies these properties.

By \eqref{eqn.SemigroupRelation}, we see that $\{P_t^M\}_{t\geqslant 0}$ is positivity preserving since $\{P_t^G\}_{t\geqslant 0}$ is positivity preserving. For any $f, h\in L^2\left(M, d\left(\pi_{\sharp}\mu^G_R\right)\right)$ we see that
\begin{align*}
&\int_M P_t^Mf \cdot hd\left(\pi_{\sharp}\mu^G_R\right) =\int_{G} \left(\left(P_t^Mf\right)\circ \pi\right) \cdot (h\circ \pi)d\mu^G_R
\\
&
=\int_{G} \left(P_t^G(f\circ \pi)\right)  \cdot (h\circ \pi)d\mu^G_R
=\int_{G} \left(P_t^G(h\circ \pi)\right)  \cdot (f\circ \pi)d\mu^G_R
\\
&
=\int_{G} \left(\left(P_t^Mh\right)\circ \pi\right) \cdot (f\circ \pi)d\mu^G_R
=\int_M P_t^Mh \cdot fd\left(\pi_{\sharp}\mu^G_R\right),
\end{align*} 
where the third equality follows from the symmetry of $\{P_t^G\}_{t\geqslant 0}$, therefore $\{P_t^M\}_{t\geqslant 0}$ is symmetric as well. 

For any $s, t>0$ we have
\begin{align*}
\left(P_t^MP_s^Mf\right)\circ \pi & =P_t^G\left(\left(P_s^Mf\right)\circ \pi\right)=P_t^G\left(P_s^G(f\circ \pi)\right)
\\
&
=P_{t+s}^G(f\circ \pi)=\left(P_{t+s}^Mf\right)\circ \pi,
\end{align*}
where the third equality is by the semigroup property of $\{P_t^G\}_{t\geqslant 0}$, so $\{P_t^M\}_{t\geqslant 0}$ has the semigroup property as well. Finally, we can use the isometry 

\begin{align*}
& \Vert P_t^Mf\Vert^2_{L^2\left(M, d\left(\pi_{\sharp}\mu^G_R\right)\right)} =\Vert \left(P_t^Mf\right) \circ \pi\Vert^2_{L^2\left(G, d\mu^{G}_R\right)}
\\
&
=\Vert P_t^G(f\circ \pi)\Vert^2_{L^2\left(G, d\mu^{G}_R\right)}
\leqslant \Vert f\circ \pi\Vert^2_{L^2\left(G, d\mu^{G}_R\right)}
\\
&
=\Vert f\Vert^2_{L^2\left(M, d\left(\pi_{\sharp}\mu^G_R\right)\right)},
\end{align*}
to see that $\{P_t^M\}_{t\geqslant 0}$ satisfies the contraction property. 

For any $f\in L^2\left(M,  d\left(\pi_{\sharp}\mu^G_R\right)\right)$ with $0\leqslant f\left( x \right)\leqslant 1$ for $\pi_{\sharp}\mu^G_R$-a.e. $x$, we have $0\leqslant f\circ\pi \leqslant 1$ for $\mu^G_R$-a.e. and $0 \leqslant \left(P_t^Mf\right) \circ \pi =P_t^G(f\circ \pi) \leqslant 1$ for $\mu^G_R$-a.e. by the Markovian property of $\{P_t^G\}_{t\geqslant 0}$, so $0 \leqslant P_t^Mf \leqslant 1$ for $\pi_{\sharp}\mu^G_R$-a.e. and $\{P_t^M\}_{t\geqslant 0}$ is Markovian.

Moreover, for any $f\in L^2\left(M, d\left(\pi_{\sharp}\mu^G_R\right)\right)$ we have
\begin{align*}
\Vert P_t^Mf-f\Vert^2_{L^2\left(M, d\left(\pi_{\sharp}\mu^G_R\right)\right)} & =\Vert \left(P_t^Mf\right) \circ \pi-f\circ \pi\Vert^2_{L^2\left(G, d\mu^{G}_R\right)}
\\
&
=\Vert P_t^G(f\circ \pi)-f\circ \pi\Vert^2_{L^2\left(G, d\mu^{G}_R\right)}
\xrightarrow[t\to 0]{} 0
\end{align*}
by the strong continuity of $\{P_t^G\}_{t\geqslant 0}$, thus $\{P_t^M\}_{t\geqslant 0}$ is strongly continuous as well.
\end{proof}
As a result, there exists a self-adjoint infinitesimal generator of $\{P_t^M\}_{t\geqslant 0}$ on $L^2\left(M, d\left(\pi_{\sharp}\mu^G_R\right)\right)$ by \cite[Lemma 1.31]{FukushimaOshimaTakedaBook2011}. This operator is a self-adjoint extension of the sub-Laplacian $\Delta_{\mathcal{H}}^M$. For simplicity, we still denote it by $\Delta_{\mathcal{H}}^M$. The domain $\mathcal{D}\left(\Delta_{\mathcal{H}}^M\right) \subseteq L^2\left(M, d\left(\pi_{\sharp}\mu^G_R\right)\right)$ of this generator is given by
\begin{align} \label{eqn.DomainofSubLaplacianonHomogeneousSpace}
& \mathcal{D}\left(\Delta_{\mathcal{H}}^M\right)
\\
& =\left\{f\in L^2\left(M, d\left(\pi_{\sharp}\mu^G_R\right)\right): \lim_{t\to 0}\frac{P_t^Mf-f}{t} \, \text{exists in} \, L^2\left(M, d\left(\pi_{\sharp}\mu^G_R\right)\right)\right\}.
\notag
\end{align}

\begin{lemma}
For any $f\in \mathcal{D}\left(\Delta_{\mathcal{H}}^{M}\right)$, we have $f\circ \pi \in \mathcal{D}\left(\Delta_{\mathcal{H}}^G\right)$ and
\begin{align} \label{eqn.LaplacianRelationII}
\left(\Delta_{\mathcal{H}}^{M}f\right) \circ \pi=\Delta_{\mathcal{H}}^G(f\circ \pi).
\end{align}
\end{lemma}

\begin{proof}
For any $f\in \mathcal{D}\left(\Delta_{\mathcal{H}}^{M}\right) \subseteq L^2\left(M, d\left(\pi_{\sharp}\mu^G_R\right)\right)$, we see that the existence of $\left(\lim_{t\to 0}\frac{P_t^Mf-f}{t}\right)\circ \pi$ in $L^2\left(G, d\mu^G_R\right)$ is guaranteed by the existence of $\lim_{t\to 0}\frac{P_t^Mf-f}{t}$ in $L^2\left(M, d\left(\pi_{\sharp}\mu^G_R\right)\right)$ using \eqref{eqn.L^2NormReferenceMeasureRelation}. This means
\begin{align*}
\left(\Delta_{\mathcal{H}}^{M}f\right) \circ \pi & =\left(\lim_{t\to 0}\frac{P_t^Mf-f}{t}\right)\circ \pi
\\
&
=\lim_{t\to 0}\left(\frac{P_t^Mf-f}{t}\circ \pi\right)=\lim_{t\to 0}\frac{\left(P_t^Mf\right)\circ \pi-f\circ \pi}{t}
\\
&
=\lim_{t\to 0}\frac{P_t^G(f\circ \pi)-(f\circ \pi)}{t}=\Delta_{\mathcal{H}}^G(f\circ \pi)
\end{align*}
in $L^2\left(G, d\mu^G_R\right)$ where the third equality is by \eqref{eqn.SemigroupRelation}. Thus, $f\circ \pi \in \mathcal{D}\left(\Delta_{\mathcal{H}}^G\right)$ and \eqref{eqn.LaplacianRelationII} follows from the above computation.
\end{proof}

By \eqref{eqn.LaplacianRelationII} and \eqref{eqn.SemigroupRelation}, we see that the semigroup $P_t^M$ satisfies the following heat equation
\begin{align*}
& \frac{d}{dt}P_t^Mf=\frac{1}{2}P_t^M\Delta^{M}_{\mathcal{H}} f,
\\
&
\lim_{t\to 0}P_t^Mf=f
\end{align*}
for any $f\in \mathcal{D}\left(\Delta_{\mathcal{H}}^{M}\right)$. Next we describe the heat kernel measure on $M$ using this heat semigroup on $M$ similarly to \cite{DriverGrossSaloff-Coste2010}.

\begin{definition} \label{df.HeatKernelMeasureonQuotientSpace}
We call the pushforward measure $\pi_{\sharp}\mu_t^{G}$ of the heat kernel measure $\mu_t^G$ on $G$ to $M$ by $\pi$ the \emph{heat kernel measure} $\mu_t^{M}$ on $M$. That is, for any $A\in\mathcal{B}\left(M\right)$
\begin{align}
\mu_t^{M}(A):=\mu_t^{G}\left(\pi^{-1}(A)\right).
\end{align}
\end{definition}
In particular, for any $A\in\mathcal{B}\left(M\right)$ we have
\begin{align*}
\mu_t^{M}(A)=P_t^M\mathbbm{1}_{A}(He).
\end{align*}
Motivated by \cite[Theorem 6.15]{DriverGrossSaloff-Coste2010}, we can characterize $\mu_t^{M}$ when $G$ is connected and $H$ is only assumed to be a closed subgroup, and we give an intrinsic description of $\mu_t^{M}$ below. In particular, this theorem proves part (2) of Theorem~\ref{thm.MainTheorem}.

\begin{theorem} 
The family of probability measures $\{\mu_t^M:t>0\}$ is the unique family of probability measures $\{\lambda_t:t>0\}$ on $M$ such that the following heat equation holds
\begin{align*}
& \frac{d}{dt} \int_{M} fd\lambda_t=\frac{1}{2}\int_{M} \Delta_{\mathcal{H}}^Mfd\lambda_t,
\\
&
\lim_{t \to 0} \int_{M} fd\lambda_t=f(He)
\end{align*}
for any $f\in \mathcal{D}\left(\Delta_{\mathcal{H}}^M\right)$.
\end{theorem}

\begin{proof}
For any $f\in \mathcal{D}\left(\Delta_{\mathcal{H}}^{M}\right)$, we have $f\circ \pi \in \mathcal{D}\left(\Delta_{\mathcal{H}}^G\right)$ and the heat equation on $G$ holds for $f\circ \pi$ as follows
\begin{align*}
& \frac{d}{dt} \int_{G} f\circ \pi d\mu_t^G=\frac{1}{2}\int_{G} \Delta_{\mathcal{H}}^G (f\circ \pi) d\mu_t^G,
\\
&
\lim_{t \to 0} \int_{G} f\circ \pi d\mu_t^G=(f\circ \pi)(e).
\end{align*}
Using again the change of variable formula for pushforward measures \cite[Theorem 1, Chapter V \S6]{BourbakiBookIntegrationI2004} under the quotient map $\pi(g)\longmapsto m\in M$, we see that
\begin{align*}
& \frac{d}{dt} \int_{M} fd\left(\pi_{\sharp}\mu_t^{G}\right)=\frac{1}{2}\int_{M} \Delta_{\mathcal{H}}^Mfd\left(\pi_{\sharp}\mu_t^{G}\right),
\\
&
\lim_{t \to 0} \int_{M} fd\left(\pi_{\sharp}\mu_t^{G}\right)=f(He)
\end{align*}
by \eqref{eqn.LaplacianRelationII}. By the uniqueness of the solution to the heat equation, we see that $\{\mu_t^M:t>0\}=\{\pi_{\sharp}\mu_t^{G}:t>0\}$ is the unique family of probability measures on $M$ such that the heat equation holds.
\end{proof}

Furthermore, for any $f\in \mathcal{D}\left(\Delta_{\mathcal{H}}^M\right)$, using the change of variable formula again and definition of $\mu_t^{M}$ in Definition~\ref{df.HeatKernelMeasureonQuotientSpace}, we have
\begin{align} \label{eqn.L^2NormRelation}
\Vert f\Vert_{L^2\left(M, d\mu_t^{M}\right)}=\Vert f\circ \pi\Vert_{L^2\left(G, d\mu_t^{G}\right)}.
\end{align}

\subsection{Dirichlet forms associated to the heat kernel measure on homogeneous spaces}
Consider the Dirichlet form with respect to the heat kernel measure $\mu_t^M$ on $M$

\begin{align}
\mathcal{E}^0_{\mu_t^M}(f, h):=\int_{M} \left\langle \nabla_{\mathcal{H}}^{M}f, \nabla_{\mathcal{H}}^{M}h\right\rangle_{\mathcal{H}} d\mu_t^M
\end{align}
for any $f, h\in C^{\infty}_c\left(M\right)$.

\begin{notation}
We denote $\mathcal{E}_{\mu_t^{G}}$ by $\mathcal{E}_{G}$ and $\mathcal{E}^0_{\mu_t^{M}}$ by $\mathcal{E}^0_{M}$ respectively.
\end{notation}

\begin{lemma} \label{lem.DomainRelation}
For any $f\in C^{\infty}_c\left(M\right)$, we have $f\circ \pi \in \mathcal{D}\left(\mathcal{E}_{G}\right)$, and for any $f, h \in C^{\infty}_c\left(M\right)$, we have
\begin{align} \label{eqn.DirichletFormRelation}
\mathcal{E}^0_{M}\left(f, h\right)=\mathcal{E}_{G}\left(f\circ\pi, h\circ\pi\right).
\end{align}
\end{lemma}

\begin{proof}
For any $f\in C^{\infty}_c\left(M\right)$, we see that $f\circ \pi\in L^2\left(G, d\mu_t^{G}\right)$ by \eqref{eqn.L^2NormRelation}. Also, from the smoothness of $f$ and the fact that $f$ has compact support, we have $\dot{X}_if \in L^2\left(M, d\mu_t^{M}\right)$ and so $X_i( f\circ \pi)=\left(\dot{X}_if\right)\circ \pi \in L^2\left(G, d\mu_t^{G}\right)$ for $i=1,\cdots,n$ using the isometry \eqref{eqn.L^2NormRelation}. Thus, $f\circ \pi \in \mathcal{D}\left(\mathcal{E}_{G}\right)$. Then we can use the change of variable formula and \eqref{eqn.GradientRelation} to obtain 
\begin{align*}
\mathcal{E}^0_{M}\left(f\right)=\mathcal{E}_{G}\left(f\circ\pi\right)
\end{align*}
which implies \eqref{eqn.DirichletFormRelation} using polarization.
\end{proof}

By \eqref{eqn.L^2NormRelation} and \eqref{eqn.DirichletFormRelation}, we have 

\begin{align} \label{eqn.DirichletNormRelation}
\Vert f\Vert_{\mathcal{E}^0_{M}}=\Vert f\circ \pi\Vert_{\mathcal{E}_{G}}
\end{align}
for any $f\in C^{\infty}_c\left(M\right)$.

The first part of Theorem~\ref{thm.DirichletFormRelation} proves part (3) of Theorem~\ref{thm.MainTheorem}.

\begin{theorem} \label{thm.DirichletFormRelation}
The bilinear form $\mathcal{E}^0_{M}$ is closable on $L^2\left(M, d\nu_t^M\right)$ and its the closure denoted by $\mathcal{E}_{M}$ with its domain $\mathcal{D}\left(\mathcal{E}_{M}\right) \subseteq L^2\left(M, d\nu_t^M\right)$ is a Dirichlet form. Moreover, 
if $f\in \mathcal{D}\left(\mathcal{E}_{M}\right)$, then $f\circ \pi \in \mathcal{D}\left(\mathcal{E}_{G}\right)$ and 
\begin{align} \label{eqn.DirichletFormRelationII}
\mathcal{E}_{M}\left(f, h\right)=\mathcal{E}_{G}\left(f\circ\pi, h\circ\pi\right)
\end{align}
for any $f, h \in \mathcal{D}\left(\mathcal{E}_{M}\right)$.
\end{theorem}

\begin{proof}
First, we show the closability of $\mathcal{E}^0_{M}$. This comes from the closedness of $\mathcal{E}_{G}$. Let $\{f_k\}_{k=1}^{\infty}$ be a sequence in $C^{\infty}_c\left(M\right)$ such that $\Vert f_k\Vert_{L^2\left(M, d\mu_t^M\right)} \xrightarrow[
k\to\infty]{} 0$ and $\mathcal{E}^0_{M}(f_k-f_l) \xrightarrow[k, l \to\infty]{}  0$. Then we have $f_k\circ \pi, f_l\circ \pi \in \mathcal{D}\left(\mathcal{E}_{G}\right)$ by Lemma~\ref{lem.DomainRelation} with
\begin{align*}
\Vert f_k\circ \pi\Vert_{L^2\left(G, d\mu_t^{G}\right)}=\Vert f_k\Vert_{L^2\left(M, d\mu_t^M\right)} \xrightarrow[k\to\infty]{} 0
\end{align*}
by \eqref{eqn.L^2NormRelation} and
\begin{align*}
\mathcal{E}_{G}(f_k\circ\pi-f_l\circ\pi)=\mathcal{E}^0_{M}(f_k-f_l) \xrightarrow[k, l \to\infty]{} 0
\end{align*}
by \eqref{eqn.DirichletFormRelation}. Using that $\mathcal{E}_{G}$ is closed, we see that
\begin{align*}
\mathcal{E}^0_{M}(f_k)=\mathcal{E}_{G}(f_k\circ\pi) \xrightarrow[k\to\infty]{} 0,
\end{align*}
and therefore $\mathcal{E}^0_{M}$ is closable on $L^2\left(M, d\mu_t^M\right)$.

Next, we show that the closure $\mathcal{E}_{M}$ of $\mathcal{E}^0_{M}$ is a Dirichlet form. It suffices to show that $\mathcal{E}^0_{M}$ is Markovian. Then, by \cite[Theorem 3.1.1]{FukushimaOshimaTakedaBook2011}, the closability of $\mathcal{E}^0_{M}$ implies that $\mathcal{E}_{M}$ is Markovian too. It remains to show that $\mathcal{E}^0_{M}$ satisfies the definition of being Markovian. For each $\varepsilon>0$, by  \cite[Exercise 1.2.1]{FukushimaOshimaTakedaBook2011} we can find an infinitely differentiable function $\phi_{\varepsilon}(t)$ such that $\phi_{\varepsilon}(t)=t$ for $t\in [0,1]$, $-\varepsilon\leqslant \phi_{\varepsilon}(t) \leqslant 1+\varepsilon$ for any $t\in \mathbb{R}$ and $0\leqslant \phi_{\varepsilon}(t^{\prime})-\phi_{\varepsilon}(t)\leqslant t^{\prime}-t$ whenever $t<t^{\prime}$. For any $f\in C^{\infty}_c(M)$, we have $\phi_{\varepsilon}(f)\in C^{\infty}_c(M)$ and $\phi_{\varepsilon}(f)\circ \pi=\phi_{\varepsilon}(f\circ \pi)\in \mathcal{D}\left(\mathcal{E}_{G}\right)$. Then, the Markovian property of $\mathcal{E}_{G}$ together with \eqref{eqn.DirichletFormRelation} implies that
\begin{align*}
\mathcal{E}^0_{M}\left(\phi_{\varepsilon}(f)\right) & =\mathcal{E}_{G}\left(\phi_{\varepsilon}(f\circ \pi)\right)
\\
&
\leqslant \mathcal{E}_{G}(f\circ \pi),
\end{align*}
so $\mathcal{E}^0_{M}$ is Markovian and thus $\mathcal{E}_{M}$ is Markovian. Thus $\mathcal{E}_{M}$ is a Dirichlet form.

Now we connect $\mathcal{D}\left(\mathcal{E}_{M}\right)$ and $\mathcal{D}\left(\mathcal{E}_{G}\right)$. For any $f\in \mathcal{D}\left(\mathcal{E}_{M}\right)$, there exists a sequence $\{f_k\}_{k=1}^{\infty}\subseteq C^{\infty}_c\left(M\right)$ such that $f_k\xrightarrow[k\to\infty]{} 
 f$ under $\Vert \cdot\Vert_{\mathcal{E}_{M}}$. Using \eqref{eqn.DirichletNormRelation} and the fact that $\mathcal{E}_{G}$ is closed, we obtain the convergence of the sequence $\{f_k\circ \pi\}_{k=1}^{\infty}$ in $\mathcal{D}\left(\mathcal{E}_{G}\right)$ under $\Vert \cdot\Vert_{\mathcal{E}_{G}}$. This means that there exists an $\tilde{f}\in \mathcal{D}\left(\mathcal{E}_{G}\right)$ such that $\Vert f_k\circ \pi\Vert_{\mathcal{E}_{G}} \xrightarrow[k\to\infty]{} \left\Vert \widetilde{f}\right\Vert_{\mathcal{E}_{G}}$. More precisely, we have
\begin{align*}
\Vert f_k\circ \pi\Vert_{L^2\left(G, d\mu_t^G\right)} \xrightarrow[k\to\infty]{} \left\Vert \widetilde{f}\right\Vert_{L^2\left(G, d\mu_t^G\right)}, \, \mathcal{E}_{G}(f_k\circ \pi) \xrightarrow[k\to\infty]{} \mathcal{E}_{G}\left(\tilde{f}\right).
\end{align*}
However, by \eqref{eqn.L^2NormRelation} we have 
\begin{align*}
\Vert f_k\circ \pi\Vert_{L^2\left(G, d\mu_t^G\right)} & =\Vert f_k\Vert_{L^2\left(M, d\mu_t^M\right)} 
\\
&
\xrightarrow[k\to\infty]{} \Vert f\Vert_{L^2\left(M, d\mu_t^M\right)}=\Vert f\circ \pi \Vert_{L^2\left(G, d\mu_t^G\right)}.
\end{align*}
By the uniqueness of the limit in $L^2\left(G, d\mu_t^G\right)$, we see that $f\circ \pi=\widetilde{f}$ in $L^2\left(G, d\mu_t^G\right)$. Since $\mathcal{D}\left(\mathcal{E}_{G}\right)$ is a subspace of $L^2\left(G, d\mu_t^G\right)$, we obtain that $f\circ \pi=\widetilde{f}$ in $\mathcal{D}\left(\mathcal{E}_{G}\right)$.

Finally, we prove \eqref{eqn.DirichletFormRelationII}. It suffices to prove it when $f=h \in \mathcal{D}\left(\mathcal{E}_{M}\right)$, from which we can obtain \eqref{eqn.DirichletFormRelationII} using polarization. Using the previous convergence of $f_k \xrightarrow[k\to\infty]{} f$ in $\mathcal{D}\left(\mathcal{E}_{M}\right)$ under $\Vert \cdot\Vert_{\mathcal{E}_{M}}$ and $f_k\circ \pi \xrightarrow[k\to\infty]{} f\circ \pi$ in $\mathcal{D}\left(\mathcal{E}_{G}\right)$ under $\Vert \cdot\Vert_{\mathcal{E}_{G}}$, we have
\begin{align*}
\mathcal{E}_{G}\left(f\circ\pi\right)=\mathcal{E}_{G}\left(\tilde{f}\right)=\lim_{k\to\infty}\mathcal{E}_{G}(f_k\circ \pi)=\lim_{k\to\infty}\mathcal{E}_{M}(f_k)=\mathcal{E}_{M}(f)
\end{align*}
implying the result.
\end{proof}

As a byproduct, we have 
\begin{align*} 
\Vert f\Vert_{\mathcal{E}_{M}}=\Vert f\circ \pi\Vert_{\mathcal{E}_{G}}
\end{align*}
for any $f\in \mathcal{D}\left(\mathcal{E}_{M}\right)$.

\subsection{Logarithmic Sobolev inequalities on homogeneous spaces} \label{sec.LSIonHomoegeneousSpaces}

In this section, we prove a logarithmic Sobolev inequality on the homogeneous space $M$.  Meanwhile, we discuss the connection between logarithmic Sobolev inequalities on $G$ and $M$. From the construction of $M$, we see that it suffices to study how the action of $H$ affects the logarithmic Sobolev constant when we want to track the dependence of the logarithmic Sobolev constant on the geometry of $M$. Theorem~\ref{thm.LSIonHomogeneousSpace} indeed shows that the logarithmic Sobolev constant can be chosen to be independent of $H$, the isotropy group of $M$. In particular, Theorem~\ref{thm.LSIonHomogeneousSpace} proves part (4) of Theorem~\ref{thm.MainTheorem}.

\begin{theorem} \label{thm.LSIonHomogeneousSpace}
Suppose that $G$ is equipped with a sub-Riemannian structure $\left(G,\mathcal{H}^{G},\langle\cdot,\cdot\rangle^{G}_{\mathcal{H}}\right)$ and the logarithmic Sobolev inequality \eqref{LSI} holds for  $f\in C^{\infty}_c(G)$ with the constant $C\left(G, \mathcal{H}^G, \mu_t^G \right)$. Then there is a natural sub-Riemannian structure $\left(M,\mathcal{H}^{M}, \langle\cdot, \cdot\rangle^{M}_{\mathcal{H}}\right)$ on $M$ induced by the transitive action by $G$ and $LSI_C\left(M,\mathcal{H}^{M}, \mathcal{D}\left(\mathcal{E}_{M}\right), \mu_t^{M}\right)$ holds. Moreover, the constant $C\left(M, \mathcal{H}^{M}, \mu_t^{M}\right)$ can be chosen to be
\begin{align*}
C\left(M,\mathcal{H}^{M},\mu_t^{M}\right)=C\left(G, \mathcal{H}^G,\mu^G_t\right).
\end{align*}
\end{theorem}

\begin{proof}
By Proposition~\ref{prop.SubRiemannianOnHomogeneousSpaces}, we see that there is a natural sub-Riemannian structure $\left(M,\mathcal{H}^{M}, \langle\cdot, \cdot\rangle^{M}_{\mathcal{H}}\right)$ on $M$ induced by the transitive action by $G$.

Now we prove a logarithmic Sobolev inequality on $M$. By Proposition~\ref{prop.LimitInLSI} and since $\mathcal{E}_{M}$ is closed, it suffices to prove that \eqref{LSI} holds for $f\in C^{\infty}_c(M) \subseteq \mathcal{D}\left(\mathcal{E}_{M}\right)$. Recall that then $f\circ \pi \in \mathcal{D}\left(\mathcal{E}_{G}\right)$ by Theorem~\ref{thm.DirichletFormRelation}. By Proposition~\ref{prop.LimitInLSI} and since $\mathcal{E}_{G}$ is closed, we have that $LSI_C(G,\mathcal{H}^G, \mathcal{D}\left(\mathcal{E}_{G}\right), \mu_t^G)$ holds on $G$. Using the change of variable formula for the quotient map $\pi(g)\longmapsto m\in M$ in the following logarithmic Sobolev inequality for $f\circ \pi$
\begin{align*} 
& \int_{G}(f\circ \pi)^2\log (f\circ \pi)^2d\mu_t^{G}-\left(\int_{G}(f\circ \pi)^2 d\mu_t^{G}\right)\log\left(\int_{G}(f\circ \pi)^2d\mu_t^{G}\right) 
\\
&
\leqslant C\left(G,\mathcal{H}^G, \mu_t^G\right)\mathcal{E}_{G}(f\circ \pi), 
\end{align*}
we obtain that 
\begin{align} \label{ineq.LSIonHomogeneousSpace}
\int_{M}f^2\log f^2d\mu_t^{M}-\left(\int_{M}f^2 d\mu_t^{M}\right)\log\left(\int_{M}f^2d\mu_t^{M}\right) \leqslant C\left(G, \mathcal{H}^G,\mu_t^G\right) \mathcal{E}_{M}(f) 
\end{align}
by \eqref{eqn.DirichletFormRelationII}. Thus $LSI_C\left(M,\mathcal{H}^{M},\mathcal{D}\left(\mathcal{E}_{M}\right),\mu_t^{M}\right)$ holds. Moreover, we can choose
\begin{align*}
C\left(M, \mathcal{H}^{M},\mu_t^{M}\right)=C\left(G, \mathcal{H}^G,\mu^G_t\right)
\end{align*}
as we can see in \eqref{ineq.LSIonHomogeneousSpace}.
\end{proof}

\begin{remark}
Theorem~\ref{thm.LSIonHomogeneousSpace} tells us that if two homogeneous spaces have a transitive action by the same Lie group $G$, then they satisfy a logarithmic Sobolev inequality with the same constant. This means that the logarithmic Sobolev constant can be chosen independent of the isotropy group of the homogeneous space, though the constant might or might not be optimal for different isotropy groups. 
In addition, it is not clear whether the quotient map $\pi$ preserves the optimality of the logarithmic Sobolev constant or not. 
\end{remark}

\begin{remark}
If $G$ is a Riemannian manifold, then $M$ is a Riemannian manifold, so Theorem~\ref{thm.LSIonHomogeneousSpace} recovers \cite[Corollary 4.5]{Gross1992}.
\end{remark}

\section{Heisenberg group and step-two homogeneous spaces} \label{sec.Examples}

The first type of examples are built on the three-dimensional isotropic Heisenberg group, which is a model space in sub-Riemannian geometry. The transitive action on such homogeneous spaces is given by a three-dimensional isotropic Heisenberg group or their product groups. We first recall some basics about it.

The three-dimensional isotropic Heisenberg group $\mathbb{H}^1_{\omega_0}$ is the set $\mathbb{R}^{2}\times\mathbb{R}$ equipped with the group law given by
\begin{align}\label{GroupLaw.n=1}
& \left( x, y, z \right)\star \left( x^{\prime}, y^{\prime}, z^{\prime} \right)=\left( x+x^{\prime}, y+y^{\prime}, z+z^{\prime} + \frac{1}{2}\left( xy^{\prime}- x^{\prime}y\right)\right)
\end{align}
for any $\left( x, y, z \right), \left( x^{\prime}, y^{\prime}, z^{\prime} \right) \in \mathbb{H}^1_{\omega_0}$.
The Lie algebra $\mathfrak{h}_{\omega_0}$ can be identified with the linear space spanned by the collection of the following left-invariant vector fields
\begin{align}\label{e.CanonicalBasis}
& \widetilde{X}\left( g \right)=\frac{\partial}{\partial x}-\frac{1}{2}y\frac{\partial}{\partial z}, \notag
\\
&
\widetilde{Y}\left( g \right)=\frac{1}{\partial y}+\frac{1}{2}x\frac{\partial}{\partial z}, \notag
\\
&
\widetilde{Z}\left( g \right)=\frac{\partial}{\partial z} \notag
\end{align}
for any $g=(x,y,z)\in\mathbb{H}^{1}_{\omega_0}$.
The isotropic Heisenberg group $\mathbb{H}^1_{\omega_{0}}$ has a natural sub-Riemannian structure $\left(\mathbb{H}^1_{\omega_{0}},\mathcal{H}^{\omega_0},\langle\cdot,\cdot\rangle^{\omega_0}_{\mathcal{H}}\right)$ where the \emph{horizontal distribution} is
\[
\mathcal{H}^{\omega_0}=\mathcal{H}_{g}^{\omega_0}= \operatorname{Span}\{\widetilde{X} \left( g \right), \widetilde{Y}\left( g \right)\}
\]
and the left-invariant inner product $\langle\cdot,\cdot\rangle_{\mathcal{H}^{\omega_0}}$ is chosen in such a way that  $\{\widetilde{X}, \widetilde{Y}\}$ is an orthonormal frame for the sub-bundle $\mathcal{H}^{\omega_0}$. The bi-invariant Haar measure $dg$ on $\mathbb{H}^1_{\omega_0}$ is the Lebesgue measure $dxdydz$ and we choose it as our reference measure on $\mathbb{H}^1_{\omega_0}$. 

Let $\mu_t^{\omega_0}$ be the hypoelliptic heat kernel measure on $\mathbb{H}^1_{\omega_{0}}$ associated to the sub-Laplacian $\Delta_{\mathcal{H}}^{\omega_0}$. The logarithmic Sobolev inequality respect to the heat kernel measure is known to hold in this case.

\begin{theorem}[Corollaire 1.2 in \cite{LiHong-Quan2006}]\label{t.HQLi} On $\mathbb{H}^1_{\omega_0}$, the logarithmic Sobolev inequality \eqref{LSI} with respect to the heat kernel measure $\mu_t^{\omega_0}$ holds for $f\in C^{\infty}_c(\mathbb{H}^1_{\omega_0})$ with a constant $C(\omega_0,t)$. 
\end{theorem}

\begin{remark}
In addition to the statement above H.-Q.~Li proved that

\[
C\left(\omega_{0}, t\right)=C\left(\omega_0\right)t=C^{2}t,
\]
where $C$ is the constant in the Driver-Melcher inequality \cite[Th\'{e}or\`{e}me 1.1]{LiHong-Quan2006} proved originally in \cite{DriverMelcher2005} for $p>1$. 
\end{remark}

For $n\geqslant 1$, let $G$ be the product group $\mathbb{H}^1_{\omega_{0}} \times \cdots \times \mathbb{H}^1_{\omega_{0}}$ of $n$ copies of $\mathbb{H}^1_{\omega_{0}}$. The product group is equipped with a product sub-Riemannian structure $\left(\mathbb{H}^1_{\omega_{0}} \times \cdots \times \mathbb{H}^1_{\omega_{0}}, \mathcal{H}, \langle \cdot, \cdot \rangle_{\mathcal{H}}\right)$ as in  \cite[Section 4.1]{GordinaLuo2022}. We know that $G$ is unimodular. The bi-invariant Haar measure $\mu^G=\mu^G_R=\mu^G_L$ on $\mathbb{H}^1_{\omega_{0}} \times \cdots \times \mathbb{H}^1_{\omega_{0}}$ is the Lebesgue measure $dx_1dy_1dz_1\cdots dx_ndy_ndz_n$ and we choose it as our reference measure on $\mathbb{H}^1_{\omega_{0}} \times \cdots \times \mathbb{H}^1_{\omega_{0}}$. 

The hypoelliptic heat kernel measure $\mu_t^{G}$ associated to the sub-Laplacian $\Delta_{\mathcal{H}}$ on the product group is the product measure $\mu_t^{\omega_{0}} \otimes\cdots\otimes \mu_t^{\omega_{0}}$ (see \cite[Section 4.1]{GordinaLuo2022} for details). The logarithmic Sobolev inequality with respect to the heat kernel measure can be obtained via a tensorization argument.

\begin{proposition} [Proposition 4.1 in \cite{GordinaLuo2022}] \label{prop.LSITensorization} 
On the product group $\mathbb{H}^1_{\omega_0}\times \cdots\times \mathbb{H}^1_{\omega_0}$, the logarithmic Sobolev inequality \eqref{LSI} with respect to the heat kernel measure $\mu_t^{G}$ holds with a constant $C\left(n,t\right)$, where the constant can be chosen to be $C\left(n,t\right)=C\left(\omega_0\right)t$ which is independent of $n$.
\end{proposition}

The homogeneous space characterization theorem, Theorem~\ref{t.HomSpaceCharacterization}, tells us that the homogeneous space is determined by a closed subgroup $H$ of $G$ as the isotropy subgroup. First we do not specify such a subgroup $H$ and present our result on hypoelliptic logarithmic Sobolev inequalities before proceeding to different examples of $H$.

Let $H$ be a closed subgroup of $\mathbb{H}^1_{\omega_1}\times \cdots\times \mathbb{H}^1_{\omega_n}$. By \cite[Proposition 10 in Chapter VII \S2]{BourbakiBookIntegrationII2004}, $H$ is unimodular too. We choose the bi-invariant Haar measure $\mu^H=\mu^H_R=\mu^H_L$ as a reference measure on $H$.

From the construction of the homogeneous space, we see that $M=H\backslash G$ is a step-two homogeneous space with a natural sub-Riemannian structure as described in Section~\ref{sec.SubRiemannianStructureonHomogoneousSpaces}. Furthermore, if $H$ is a normal subgroup of $\mathbb{H}^1_{\omega_1}\times \cdots\times \mathbb{H}^1_{\omega_n}$, then $M$ is indeed a step-two Carnot group (see \cite[Section 2.2]{BonfiglioliLanconelliUguzzoniBook} for precise definition). 

We give a characterization of closed subgroups of any step-two Carnot group as below. This result applies to $\mathbb{H}^1_{\omega_{0}} \times \cdots \times \mathbb{H}^1_{\omega_{0}}$ too.

\begin{lemma} \label{eqn.ClosedLieSubgroup}
Let $G$ be a step-two Carnot group with the Lie algebra $\mathfrak{g}=V_1\oplus V_2$ such that $[V_1,V_1]=V_2$. Let $\mathfrak{h}=\widetilde{V_1}\oplus \widetilde{V_2}$ where $\widetilde{V_1}$ is a subspace of $V_1$ and $\widetilde{V_2}$ is a subspace of $V_2$ such that $[\widetilde{V_1},\widetilde{V_1}] \subseteq \widetilde{V_2}$. Then $\mathfrak{h}$ is a Lie subalgebra of $\mathfrak{g}$ and $H=\operatorname{exp}\left(\mathfrak{h}\right)$ is a closed subgroup of $G$. Conversely, every closed subgroup of $G$ has such an explicit form.
\end{lemma}

\begin{proof}
On the one hand, from the construction of $\mathfrak{h}$, we see that $\mathfrak{h}$ is a Lie subalgebra of $\mathfrak{g}$. In addition, the inclusion map $\iota:H \hookrightarrow G$ is an embedding. By \cite[Theorem 7.21]{LeeBook2003SmoothManifold}, $H$ is a closed subgroup.

On the other hand, for any closed subgroup $H$ of $G$, by \cite[Theorem 20.12]{LeeBook2003SmoothManifold}, $H$ is  an embedded
Lie subgroup. Let $\mathfrak{h}$ be the Lie algebra of $H$. Then $\mathfrak{h}$ is a Lie subalgebra of $\mathfrak{g}$ up to some identification (see \cite[p.47]{KnappBook1996}). This enables us to write $\mathfrak{h}$ as $\mathfrak{h}=\widetilde{V_1}\oplus \widetilde{V_2}$ where $\widetilde{V_i}$ is a subspace of $V_i$ for $i=1,2$. The fact that $[\mathfrak{h},\mathfrak{h}] \subseteq \mathfrak{h}$ implies  that $[\widetilde{V_1},\widetilde{V_1}] \subseteq \widetilde{V_2}$. Using the fact that the exponential map of a Carnot group is a global diffeomorphism (see \cite[Theorem 1.3.28]{BonfiglioliLanconelliUguzzoniBook}), we have $H=\operatorname{exp}\left(\mathfrak{h}\right)$ by \cite[Proposition 6.3]{SagleWaldeBook1973}.
\end{proof}

Theorem~\ref{thm.LSIonHomogeneousSpace} allows us to study the dependence of the hypoelliptic logarithmic Sobolev constant of $M$ on its underlying geometry and the dimension as follows.

\begin{theorem} \label{thm.LSIonHomogeneousSpaceII}
There is a natural sub-Riemannian structure $\left(M, \mathcal{H}^{M}, \langle\cdot, \cdot\rangle^{M}_{\mathcal{H}}\right)$ on $M$ induced by the transitive action by $\mathbb{H}^1_{\omega_0}\times \cdots\times \mathbb{H}^1_{\omega_0}$, and the logarithmic Sobolev inequality $LSI_C\left(M,\mathcal{H}^{M}, \mathcal{D}\left(\mathcal{E}_{M}\right), \mu_t^{M}\right)$ holds with the constant $C\left(M,\mathcal{H}^{M},\mu_t^{M}\right)$. Moreover, we can choose
\begin{align*}
C\left(M,\mathcal{H}^{M},\mu_t^{M}\right)=C(\omega_0,t)=C\left(\omega_0\right)t,
\end{align*}
which is the same as the constant for the isotropic Heisenberg group $\mathbb{H}^1_{\omega_{0}}$. Thus the logarithmic Sobolev constant $C\left(M,\mathcal{H}^{M},\mu_t^{M}\right)$ on the $M$ is independent of its dimension.
\end{theorem}
For the rest of this section, we consider some concrete examples of such homogeneous spaces. We specify their corresponding isotropy subgroups as a closed Lie subgroup $H$ of $G$ and their identification with known quotient spaces. This enables us to prove a dimension-independent logarithmic Sobolev inequality on these spaces.

\subsection{The real line (a non-example)} \label{Section.Heisenberg-Euclidean}

The real line $\mathbb{R}^1$ can be described as a homogeneous space under the transitive action of the three-dimensional isotropic Heisenberg group $\mathbb{H}^1_{\omega_0}$.  In the rest of this section, we identify such a homogeneous space with a quotient space and present a hypoelliptic logarithmic Sobolev inequality there.

Let $G$ be the three-dimensional isotropic Heisenberg group $\mathbb{H}^1_{\omega_0}$. We take $H=\{(x,0,z):x,z
\in\mathbb{R}\}$ and then $H$ is a normal subgroup of $\mathbb{H}^1_{\omega_0}$. The bi-invariant Haar measure $\mu^H=\mu^H_R=\mu^H_L$ on $H$ is the Lebesgue measure $dxdz$ and we choose it as our reference measure on $H$. The left action of $H$ on $G$
\begin{align}
& H \times G \rightarrow G \notag
\\
&
\left((\tilde{x},0,\tilde{z}),(x,y,z)\right) \mapsto (\tilde{x},0,\tilde{z})\circ (x,y,z)=\left(x+\tilde{x},y,z+\tilde{z}+\frac{1}{2}\tilde{x}y\right) \notag
\end{align}
is proper but \emph{not} transitive. However, this action is not isometric in the sub-Riemannian sense. Under the action $H$, for any $g=(x,y,z)\in \mathbb{H}^1_{\omega_0}$ and any $h=(\tilde{x},0,\tilde{z})\in H$, the orthonormal frame $\{\widetilde{X} \left( g \right), \widetilde{Y}\left( g \right)\}$ is mapped to $\{\widetilde{X} \left( g \right), Y-\frac{\widetilde{x}}{2}\frac{\partial}{\partial z}=\frac{\partial}{\partial y}+\frac{x-\widetilde{x}}{2}\frac{\partial}{\partial z}\}$, which is not an orthornomal frame of $\left(\mathbb{H}^1_{\omega_{0}},\mathcal{H}^{\omega_0},\langle\cdot,\cdot\rangle^{\omega_0}_{\mathcal{H}}\right)$.

The homogeneous space $M$ is a Lie group and it is isomorphic to $\mathbb{R}^1$. In this case, the induced sub-Riemannian structure on $M$ is indeed a Riemannian structure. 

For simplicity, we will identify $M$ with $\mathbb{R}^1$ in the rest of this section. We can choose the Lebesgue measure on $\mathbb{R}^1$ as our reference measure. Also, we obtain the following Lie group homomorphism
\begin{align}
& \varphi:\mathbb{H}^1_{\omega_{0}} \rightarrow \mathbb{R}\cong M, \notag
\\
& \varphi(g):=\varphi(x, y, z)=u=y. \notag
\end{align}

In this case, $\Delta_{\mathcal{H}}^{M}$ is the the Laplacian on $\mathbb{R}^1$. In addition, $\mu_t^G$ is hypoelliptic heat kernel measure $\mu_t^{\omega_0}$ on $\mathbb{H}^1_{\omega_{0}}$ and $\mu_t^{M}$ is the elliptic heat kernel measure $\mu_t$ associated to the Laplacian on $\mathbb{R}^1$.

Applying Theorem~\ref{thm.LSIonHomogeneousSpaceII}, we recover the logarithmic Sobolev inequality on $\mathbb{R}^1$ with respect to the heat kernel (Gaussian) measure and we see that we can choose the constant to be $C(\omega_0,t)=C\left(\omega_0\right)t$.

\subsection{The Grushin plane} \label{sec.Grushin}

The Grushin plane is a homogeneous space under the transitive action of $\mathbb{H}^1_{\omega_{0}}$. For the rest of this section, we identify this homogeneous space with a quotient space and present a hypoelliptic logarithmic Sobolev inequality there.

We take  $G=\mathbb{H}^1_{\omega_0}$ as in Section~\ref{Section.Heisenberg-Euclidean}. Let $H=\{(0,y,0):y
\in\mathbb{R}\}$ and it is a not normal subgroup of $\mathbb{H}^1_{\omega_0}$. The bi-invariant Haar measure $\mu^H=\mu^H_R=\mu^H_L$ on $H$ is the Lebesgue measure $dy$ and we choose it as our reference measure on $H$.

The left action of $H$ on $G$
\begin{align}
& H \times G \rightarrow G \notag
\\
&
\left((0,\tilde{y},0),(x,y,z)\right) \mapsto (0,\tilde{y},0)\circ (x,y,z)=\left(x,y+\tilde{y},z-\frac{1}{2}\tilde{x}y\right) \notag
\end{align}
is proper but not transitive. This action is isometric in the sub-Riemannian sense, that is, under the action $H$, an orthonormal frame of $\left(\mathbb{H}^{1}_{\omega_0}, \mathcal{H}^{\omega_0}, \langle \cdot, \cdot \rangle^{\omega_0}_{\mathcal{H}}\right)$ is mapped to an orthonormal frame.

The homogeneous space $M$ is no 
longer a Lie group but it is topologically isomorphic to $\mathbb{R}^2$. For simplicity, we will identify $M$ with $\mathbb{R}^2$ in the rest of this section. In this case the quasi-invariant measure $\mu^M$ under the action of $G$ is the Lebesgue measure on $\mathbb{R}^2$ by \cite[p. 464]{DriverGrossSaloff-Coste2010}. But Popp's measure and Hausdorff measure are $\frac{1}{\vert u\vert}dudv$, which do not coincide with $\mu^M$. We choose the Lebesgue measure on $\mathbb{R}^2$ as our reference measure. Also, we have the following smooth map
\begin{align}
& \varphi:\mathbb{H}^1_{\omega_{0}} \rightarrow \mathbb{R}^2\cong M \notag
\\
& \varphi(g):=\varphi(x, y, z)=(u,v)=\left(x,z+\frac{1}{2}xy\right). \notag
\end{align}

What is special about the geometry of $M$ is that $M$ has a singular Riemannian structure. In this case, 
\begin{align*}
d\varphi_g(\widetilde{X})=\frac{\partial}{\partial u}, \, d\varphi_g(\widetilde{Y})=u\frac{\partial}{\partial v}.
\end{align*}
We see that $d\varphi_g(\widetilde{X})$ and $d\varphi_g(\widetilde{Y})$ are linearly independent except along the line $u=0$. We consider a singular metric on $\mathbb{R}^2$ in such a way that $\left\{d\varphi_g(\widetilde{X}), d\varphi_g(\widetilde{Y})\right\}$ is an orthonormal frame except on the line $u=0$. 

The plane $\mathbb{R}^2$ equipped with such a singular Riemannian structure is called the \emph{Grushin plane}. The operator $\Delta_{\mathcal{H}}^{M}$ has the form
\begin{align*}
\Delta_{\mathcal{H}}^{M}=\frac{\partial^2}{\partial u^2}+u^2\frac{\partial^2}{\partial v^2}
\end{align*}
and it is called the \emph{Grushin operator}. For more details, we refer to \cite[Section 10.3]{CalinChangFurutaniIwasakiBook2011}.

In this case, $\mu_t^G$ is the hypoelliptic heat kernel measure $\mu_t^{\omega_0}$ on $\mathbb{H}^1_{\omega_{0}}$ and $\mu_t^{M}$ is the heat kernel measure on the Grushin plane as defined in Definition~\ref{df.HeatKernelMeasureonQuotientSpace}.

By Theorem~\ref{thm.LSIonHomogeneousSpaceII}, we see that 

\begin{proposition}
On the Grushin plane, there is a natural sub-Riemannian (indeed singular Riemannian) structure $\left(\mathbb{R}^2,\mathcal{H}^{\mathbb{R}^2},\langle\cdot,\cdot\rangle^{\mathbb{R}^2}_{\mathcal{H}}\right)$ induced by the transitive action by $\mathbb{H}^1_{\omega_0}$. Then, $\operatorname{LSI}_C\left(\mathbb{R}^2,\mathcal{H},\mathcal{D}\left(\mathcal{E}_{M}\right),\mu_t^{M}\right)$ holds and the logarithmic Sobolev constant can be chosen to be $C(\omega_0,t)=C\left(\omega_0\right)t$, which is the same as the logarithmic Sobolev constant on $\mathbb{H}^1_{\omega_0}$.
\end{proposition}

\subsection{Non-isotropic Heisenberg groups} \label{sec.ExampleNonisotropic}
A  non-isotropic Heisenberg group can be described as a homogeneous space under the transitive action of $\mathbb{H}^1_{\omega_{0}} \times \cdots \times \mathbb{H}^1_{\omega_{0}}$. For the rest of this section, we identify such a homogeneous space with a quotient space and present a hypoelliptic logarithmic Sobolev inequality.

Let $G$ be the product group $\mathbb{H}^1_{\omega_{0}} \times \cdots \times \mathbb{H}^1_{\omega_{0}}$ of $n$ copies of three-dimensional isotropic Heisenberg groups. For any $0<\alpha_{1}\leqslant \alpha_{2}\leqslant\cdots\leqslant \alpha_p=\alpha_{p+1}=\cdots=\alpha_{n}$, we take $H=\{\left((0,0,z_1),\cdots,(0,0,z_n)\right):z_1,\cdots,z_n
\in\mathbb{R},\sum_{i=1}^n\alpha_iz_i=0\}$ and then $H$ is a normal subgroup of $\mathbb{H}^1_{\omega_{0}} \times \cdots \times \mathbb{H}^1_{\omega_{0}}$. 

The left action of $H$ on $G$
\begin{align*}
& H \times G \rightarrow G 
\\
&
\left((\widetilde{g_1},\cdots,\widetilde{g_n}),(g_1,\cdots,g_n)\right) \longmapsto (g_1,\cdots, g_n) \circ (\widetilde{g_1},\cdots,\widetilde{g_n})
\\
& =\left((\mathbf{v}_1,z_1+\widetilde{z_1}),\cdots,(\mathbf{v}_n,z_n+\widetilde{z_n})\right) 
\end{align*}
is proper but not transitive. This action is isometric in the sub-Riemannian sense, that is, under the action $H$ an orthonormal frame is mapped to an orthonormal frame of $\left(\mathbb{H}^{1}_{\omega_0}\times\cdots\times \mathbb{H}^{1}_{\omega_0}, \mathcal{H}, \langle \cdot, \cdot \rangle_{\mathcal{H}}\right)$.

The homogeneous space $M$ is a Lie group and it is isomorphic to a $(2n+1)$-dimensional non-isotropic Heisenberg group $\mathbb{H}^n_{\omega}$ as defined in \cite[Definition 1.1]{GordinaLuo2022}. For simplicity, we will identify $M$ with $\mathbb{H}^n_{\omega}$ in the rest of this section. We choose the bi-invariant Haar measure which is the Lebesgue measure on $\mathbb{H}^n_{\omega}$ as our reference measure. Also, we obtain the following Lie group homomorphism
\begin{align}
& \pi_{\omega}:\mathbb{H}^1_{\omega_{0}} \times\cdots\times \mathbb{H}^1_{\omega_{0}} \rightarrow \mathbb{H}^{n}_{\omega}\cong M \notag
\\
& \pi_{\omega}(g_1,\cdots,g_n):=\pi_{\omega}(x_{1}, y_{1},z_{1},\cdots,x_{n},y_{n},z_{n})=(x_{1},y_{1},\cdots,z_{2},y_{2}, z),
\\
&
z=\sum_{i=1}^n\alpha_iz_i \notag
\end{align}
from \cite[Section 5]{GordinaLuo2022}. Also \cite{BaudoinGordinaSarkar2023} shows that $\varphi$ is a submersion. 

In this case,  $\Delta_{\mathcal{H}}^{M}$ is the sub-Laplacian $\Delta_{\mathcal{H}}^{\omega}$ on $\mathbb{H}^n_{\omega}$ and $\mu_t^{M}$ is the heat kernel measure $\mu_t^{\omega}$ on $\mathbb{H}^n_{\omega}$, which agrees with \cite[Proposition 5.1]{GordinaLuo2022}.  Applying Theorem~\ref{thm.LSIonHomogeneousSpaceII}, we have

\begin{corollary}[Theorem 4.5 in \cite{GordinaLuo2022}]
There is a natural sub-Riemannian structure $\left(\mathbb{H}^n_{\omega}, \mathcal{H}^{\omega}, \langle\cdot,\cdot\rangle^{\omega}_{\mathcal{H}}\right)$ on $\mathbb{H}^n_{\omega}$ induced by the transitive action by $\mathbb{H}^1_{\omega_{0}} \times\cdots\times \mathbb{H}^1_{\omega_{0}}$.
Then $\operatorname{LSI}_C(\mathbb{H}^n_{\omega},\mathcal{H}^{\omega},\mathcal{D}\left(\mathcal{E}_{\mathbb{H}^n_{\omega}}\right),\mu_t^{\omega})$ holds, and the logarithmic Sobolev constant can be chosen to be $C(\omega_0,t)=C\left(\omega_0\right)t$, which is the same as the logarithmic Sobolev constant on $\mathbb{H}^1_{\omega_0}$.
\end{corollary}

\subsection{The Heisenberg-like group} \label{sec.ExampleHeisenbergType}

The Heisenberg-like group is a homogeneous space under the transitive action of $\mathbb{H}^1_{\omega_{0}} \times \cdots \times \mathbb{H}^1_{\omega_{0}}$. For the rest of this section, we identify this homogeneous space with a quotient space and present a hypoelliptic logarithmic Sobolev inequality.

Let $G$ be the product group $\mathbb{H}^1_{\omega_{0}} \times \cdots \times \mathbb{H}^1_{\omega_{0}}$ of $n$ copies of three-dimensional isotropic Heisenberg groups as in Section~\ref{sec.ExampleNonisotropic}.

For any nonzero $\alpha_i^k$ where $i=1,\cdots,n$ and $k=1,\cdots,m$ with $m\in\mathbb{N}^{+}$, we take $H=\{\left((0,0,z_1),\cdots,(0,0,z_n)\right):\sum_{i=1}^{n} \alpha^j_iz_i=0,j=1,\cdots,m,z_i
\in\mathbb{R},i=1,\cdots,n\}$ and then $H$ is a normal subgroup of $\mathbb{H}^1_{\omega_{0}} \times \cdots \times \mathbb{H}^1_{\omega_{0}}$. 

The left action of $H$ on $G$
\begin{align*}
& H \times G \rightarrow G 
\\
&
\left((\widetilde{g_1},\cdots,\widetilde{g_n}),(g_1,\cdots,g_n)\right) \longmapsto (g_1,\cdots,g_n) \circ (\widetilde{g_1},\cdots,\widetilde{g_n})
\\
& =\left((\mathbf{v}_1,z_1+\widetilde{z_1}),\cdots,(\mathbf{v}_n,z_n+\widetilde{z_n})\right)
\end{align*}
is proper but not transitive. This action is isometric in the sub-Riemannian sense, that is, under the action $H$, an orthonormal frame is mapped to an orthonormal frame of $\left(\mathbb{H}^{1}_{\omega_0}\times\cdots\times \mathbb{H}^{1}_{\omega_0}, \mathcal{H}, \langle \cdot, \cdot \rangle_{\mathcal{H}}\right)$.

The homogeneous space $M$ is a Lie group and it is isomorphic to a Heisenberg-like group $\mathbb{H}^{(n,m)}_{\omega}$. For simplicity, we will identify $M$ with $\mathbb{H}^{(n,m)}_{\omega}$ for the rest of this section. 

\begin{definition}
A \emph{Heisenberg-like group} $\mathbb{H}^{(n,m)}_{\omega}$ is the set $\mathbb{R}^{2n}\times\mathbb{R}^m$ equipped with the group law given by
\begin{align}
& \left( \mathbf{v}, \mathbf{z} \right)\star \left( \mathbf{v}^{\prime}, \mathbf{z}^{\prime} \right)=\left(\mathbf{v}+\mathbf{v}^{\prime},\mathbf{z}+\mathbf{z}^{\prime}+\frac{1}{2}\omega\left( \mathbf{v}, \mathbf{v}^{\prime} \right)\right), \label{GroupLaw2}
\\
& \mathbf{v}=\left( x_{1},y_{1},\cdots,x_{n},y_{n}  \right), \mathbf{v}^{\prime}=\left( x_{1}^{\prime},y_{1}^{\prime},\cdots,x_{n}^{\prime},y_{n}^{\prime} \right) \in \mathbb{R}^{2n},
\notag
\\
&
\mathbf{z}=\left( z_{1},\cdots,z_{m}  \right), \mathbf{z}^{\prime}=\left( z_{1}^{\prime},\cdots,z_{m}^{\prime} \right) \in \mathbb{R}^{m}, \notag
\\
& \omega: \mathbb{R}^{2n} \times \mathbb{R}^{2n} \longrightarrow \mathbb{R}^m
\notag
\end{align}
and $\omega$ is a non-degenerate skew-symmetric bilinear form.
\end{definition}

\begin{notation}
Note that $\omega$ contains information on $m$ and $n$. To abuse notation, we will use $\mathbb{H}_{\omega}$ instead of $\mathbb{H}_{\omega}^{(n,m)}$ later on.
\end{notation}
For each $k=1,\cdots,m$, we define $\omega_k:\mathbb{R}^{2n} \times \mathbb{R}^{2n} \longrightarrow \mathbb{R}$ as $\omega_k\left(\mathbf{v},\mathbf{v}^{\prime}\right)=z_k$ for any $\mathbf{v},\mathbf{v}^{\prime}\in\mathbb{R}^{2n}$ with $\omega\left(\mathbf{v},\mathbf{v}^{\prime}\right)=\left(z_1,\cdots,z_m\right)\in \mathbb{R}^m$. Note that $\omega_k:\mathbb{R}^{2n} \times \mathbb{R}^{2n} \longrightarrow \mathbb{R}$ is a symplectic form on $\mathbb{R}^{n}$ and thus it has the explicit form
\begin{align*}
\omega_k\left( \mathbf{v}, \mathbf{v}^{\prime} \right):=\sum_{i=1}^{n} \alpha^k_i\left(x_{i}y_{i}^{\prime}-x_{i}^{\prime}y_{i}\right)
\end{align*}
where $\alpha^k_{1}, \alpha^k_{2}, \cdots, \alpha^k_{n}$ are nonzero  constants. Then \eqref{GroupLaw2} has the explicit form below
\begin{align}
\left( \mathbf{v}, \mathbf{z} \right)\star \left( \mathbf{v}^{\prime}, \mathbf{z}^{\prime} \right)=\left(\mathbf{v}+\mathbf{v}^{\prime},z_1+z_1^{\prime}+\frac{1}{2}\omega_1\left( \mathbf{v}, \mathbf{v}^{\prime} \right),\cdots,z_m+z_m^{\prime}+\frac{1}{2}\omega_m\left( \mathbf{v}, \mathbf{v}^{\prime} \right)\right), \label{GroupLaw3}
\end{align}
which coincides with \cite[Definition 3.6.1]{BonfiglioliLanconelliUguzzoniBook}. When $m=1$ and $0 \leqslant \alpha^1_1 \leqslant \alpha^1_2 \leqslant \cdots \leqslant \alpha^1_n$, we get the non-isotropic Heisenberg group.

We choose the bi-invariant Haar measure which is the Lebesgue measure on $\mathbb{H}_{\omega}$ as our reference measure. Also, we obtain the following Lie group homomorphism
\begin{align*}
& \pi_{\omega}:\mathbb{H}^1_{\omega_{0}} \times \cdots \times \mathbb{H}^1_{\omega_{0}} \rightarrow \mathbb{H}_{\omega}\cong M,
\\
& \pi_{\omega}(g_1,\cdots,g_n):=\pi_{\omega}(x_{1},y_{1},z_{1},\cdots,x_{n},y_{n},z_{n})
\\
& =(x_{1},y_{1},\cdots,x_{n},y_{n},\sum_{i=1}^{n} \alpha^1_iz_i,\cdots, \sum_{i=1}^{n} \alpha^m_iz_i)
\end{align*}
for any $(g_1,\cdots,g_n)\in \mathbb{H}^1_{\omega_{0}} \times \cdots \times \mathbb{H}^1_{\omega_{0}}$.

In this case, $\Delta_{\mathcal{H}}^{M}$ is the sub-Laplacian $\Delta_{\mathcal{H}}^{\omega}$ on $\mathbb{H}_{\omega}$ and $\mu_t^{M}$ is the heat kernel measure $\mu_t^{\omega}$ on $\mathbb{H}_{\omega}$. Applying Theorem~\ref{thm.LSIonHomogeneousSpaceII} gives us the following result. 

\begin{proposition}
On the Heisenberg-like group $\mathbb{H}_{\omega}$, there is a natural sub-Riemannian structure $\left(\mathbb{H}_{\omega},\mathcal{H}^{\omega},\langle\cdot,\cdot\rangle^{\omega}_{\mathcal{H}}\right)$ induced by the transitive action by $\mathbb{H}^1_{\omega_{0}} \times\cdots\times \mathbb{H}^1_{\omega_{0}}$. Then, $\operatorname{LSI}_C(\mathbb{H}_{\omega},\mathcal{H}^{\omega},\mathcal{D}\left(\mathcal{E}_{\mathbb{H}_{\omega}}\right),
\mu_t^{\omega})$ holds, and the logarithmic Sobolev constant can be chosen to be $C(\omega_0,t)=C\left(\omega_0\right)t$, which is the same as the logarithmic Sobolev constant on $\mathbb{H}^1_{\omega_0}$. Therefore the logarithmic Sobolev constant $C\left(\omega,t\right)=C\left(\omega_0,t\right)=C\left(\omega_0\right)t$ is independent of $\omega$, $m$ and $n$, and therefore of the dimension of $\mathbb{H}_{\omega}$. 
\end{proposition}

\section{Compact Heisenberg manifolds} \label{sec.ExamplesNilmanifold}

Many \emph{compact nilmanifolds} can be regarded as homogeneous spaces under the transitive action of a connected nilpotent Lie group $G$. Such a homogeneous space $M$ is not necessarily a Lie group. In this case, we take the isotropy subgroup $H$ to be a discrete lattice subgroup of $G$. Note that different choices of $H$ in this case can affect the topology of $M$ significantly. For example, depending on $H$, the fundamental group of $M$ could be completely different, e.~g. \cite[Corollary 2.5]{GordonWilson1986}.

A compact nilmanifold $M$ has a sub-Riemannian structure $\left(M,\mathcal{H}^{M},\langle\cdot,\cdot\rangle^{M}_{\mathcal{H}}\right)$ as described in Section~\ref{sec.SubRiemannianStructureonHomogoneousSpaces}. We denote by $\mu_t^{G}$ the hypoelliptic heat kernel measure on $G$ associated to the sub-Laplacian $\Delta_{\mathcal{H}}^{G}$. Then $\Delta_{\mathcal{H}}^{M}$ is the sub-Laplacian on $M$ and $\mu_t^{M}$ is the heat kernel measure on $M$. 

A very important class of compact nilmaniolds is the \emph{compact Heisenberg manifolds}. That is, we take $G$ to be an $(2n+1)$-dimensional isotropic Heisenberg group and $H$ to be a discrete lattice subgroup. In this case, the homogeneous space $M$ has a natural strongly pseudo-convex CR structure and can be realized as the boundary of a bounded domain in a line bundle over an Abelian variety. e.~g. \cite{Folland2004}.

Theorem~\ref{thm.LSIonHomogeneousSpace} allows us to prove a hypoelliptic logarithmic Sobolev inequality on $M$ and show that the constant is independent of its dimension.

\begin{theorem} \label{thm.LSIonHeisenbergManifold}
There is a natural sub-Riemannian structure $\left(M, \mathcal{H}^{M}, \langle\cdot, \cdot\rangle^{M}_{\mathcal{H}}\right)$ on the compact Heisenberg manifold $M$ induced by the transitive action by the $(2n+1)$-dimensional isotropic Heisenberg group, and the logarithmic Sobolev inequality $LSI_C\left(M,\mathcal{H}^{M},\mathcal{D}\left(\mathcal{E}_{M}\right),\mu_t^{M}\right)$ holds with the constant $C\left(M,\mathcal{H}^{M},\mu_t^{M}\right)$. Moreover, we can choose
\begin{align*}
C\left(M,\mathcal{H}^{M},\mu_t^{M}\right)=C(\omega_0,t)=C\left(\omega_0\right)t,
\end{align*}
which is the same as the constant for the isotropic Heisenberg group $\mathbb{H}^1_{\omega_{0}}$. Thus the logarithmic Sobolev constant $C\left(M,\mathcal{H}^{M},\mu_t^{M}\right)$ on $M$ is independent of its dimension.
\end{theorem}

Moreover, note that the logarithmic Sobolev constant on any compact Heisenberg manifold can be chosen to be the same as those of some non-compact homogeneous spaces that we considered in Section~\ref{sec.Examples}.

\section{$\operatorname{SU}(2)$ and related quotient spaces} \label{sec.ExampleSU(2)Type}

Another type of examples are built on $\operatorname{SU}(2)$, which is another model space in sub-Riemannian geometry. The transitive action on such homogeneous spaces is by $\operatorname{SU}(2)$ or their product groups.

The Lie group $\operatorname{SU}(2)$ is the group of $2\times 2$ complex unitary
matrices of determinant $1$, i.e.
\begin{align*}
G=\operatorname{SU}(2)=
\left\{\begin{pmatrix}
z_1 & z_2 \\
-\overline{z_2} & \overline{z_1}
\end{pmatrix}:z_1,z_2\in \mathbb{C},\vert z_1\vert^2+\vert z_2\vert^2=1 \right\}.   
\end{align*}
Its Lie algebra $\mathfrak{su}(2)$ consists of $2\times 2$ complex skew-adjoint
matrices with trace $0$. A basis of $\mathfrak{su}(2)$ is formed by the Pauli matrices
\begin{align*}
X=\begin{pmatrix}
0 & 1 \\ -1 & 0
\end{pmatrix}, Y=\begin{pmatrix}
0 & i \\ i & 0
\end{pmatrix}, Z=\begin{pmatrix}
i & 0 \\ 0 & -i
\end{pmatrix},
\end{align*}
for which the following relationships hold
\begin{align*}
[X,Y]=2Z,[Y,Z]=2X,[Z,X]=2Y.
\end{align*}
We denote $\widetilde{X},\widetilde{Y},\widetilde{Z}$ the left-invariant vector fields on $\operatorname{SU}(2)$ corresponding to the Pauli
matrices $X, Y, Z$. Then $\operatorname{SU}(2)$ can be equipped with a natural sub-Riemannian structure $\left(\operatorname{SU}(2),\mathcal{H},\langle\cdot,\cdot\rangle_{\mathcal{H}}\right)$ where $\mathcal{H}_g=\operatorname{Span}\{\widetilde{X}(g),\widetilde{Y}(g)\}$ at any $g\in \operatorname{SU}(2)$ and $\{\widetilde{X},\widetilde{Y}\}$ forms an orthonormal frame for $\langle\cdot,\cdot\rangle_{\mathcal{H}}$. The sub-Laplacian on $\operatorname{SU}(2)$ has the form $\Delta_{\mathcal{H}}^{\operatorname{SU}(2)}=(\widetilde{X})^2+(\widetilde{Y})^2$.

Let $\mu_t^{\operatorname{SU}(2)}$ be the hypoelliptic heat kernel measure on $\operatorname{SU}(2)$ associated to the sub-Laplacian $\Delta_{\mathcal{H}}^{SU(2)}$. The logarithmic Sobolev inequality respect to the heat kernel measure is known to hold in this case (see \cite[Theorem 1.4 and p. 2651]{BaudoinBonnefont2012}).

For $n\geqslant 1$, let $G$ be the product group $\operatorname{SU}(2) \times \cdots \times \operatorname{SU}(2)$ of $n$ copies of $\operatorname{SU}(2)$. The product group has a product sub-Riemannian structure $\left(\operatorname{SU}(2) \times \cdots \times \operatorname{SU}(2),\mathcal{H}, \langle \cdot, \cdot \rangle_{\mathcal{H}}\right)$. The group $G$ is unimodular, and therefore we choose the bi-invariant Haar measure $\mu^G=\mu^G_R=\mu^G_L$ on $\operatorname{SU}(2) \times \cdots \times \operatorname{SU}(2)$ as our reference measure on $\operatorname{SU}(2) \times \cdots \times \operatorname{SU}(2)$. 

The hypoelliptic heat kernel measure $\mu_t^{G}$ associated to the sub-Laplacian $\Delta_{\mathcal{H}}$ on the product group is the product measure $\mu_t^{\operatorname{SU}(2)} \otimes\cdots\otimes \mu_t^{\operatorname{SU}(2)}$. The logarithmic Sobolev inequality with respect to the heat kernel measure can be obtained via a tensorization argument, which was first used in \cite{Gross1975c}. More convenient to our setting is \cite[Proposition 18]{Schechtman2003a} giving the following result for the product group $\operatorname{SU}(2)\times\cdots\times  \operatorname{SU}(2)$. 

Before we present the result, we would like to discuss which function space we work with on the product group $\operatorname{SU}(2)\times\cdots\times  \operatorname{SU}(2)$. In 
\cite{BaudoinBonnefont2012},  the logarithmic Sobolev inequality was proved for functions in $C^{\infty}_c\left(\operatorname{SU}(2)\right)$. The tensorization argument allows us to prove the inequality for functions in $C^{\infty}_c\left(\operatorname{SU}(2)\right)\otimes \cdots \otimes C^{\infty}_c\left(\operatorname{SU}(2)\right)$, which is dense in $\mathcal{D}\left(\mathcal{E}_{\operatorname{SU}(2)\times \cdots\times \operatorname{SU}(2)}\right)$. Then the consideration of functions can be extended to $\mathcal{D}\left(\mathcal{E}_{\operatorname{SU}(2)\times \cdots\times \operatorname{SU}(2)}\right)$ thanks to the closedness of $\mathcal{E}_{\operatorname{SU}(2)\times \cdots\times \operatorname{SU}(2)}$. 

\begin{proposition} \label{prop.LSITensorizationII} 
On the product group $\operatorname{SU}(2)\times \cdots\times \operatorname{SU}(2)$, the logarithmic Sobolev inequality \eqref{LSI} with respect to the heat kernel measure $\mu_t^{G}$ holds with a constant $C\left(n,t\right)$, where the constant can be chosen to be $C\left(n,t\right)=C\left(\operatorname{SU}(2),t\right)$ which is independent of $n$.
\end{proposition}

The homogeneous space characterization theorem, Theorem~\ref{t.HomSpaceCharacterization}, tells us that the homogeneous space we consider here is determined by a closed subgroup $H$ of $G$ as its isotropy subgroup. We first do not specify the subgroup $H$  and present our result on hypoelliptic logarithmic Sobolev inequalities. 

Let $H$ be a closed subgroup of $\operatorname{SU}(2)\times \cdots\times \operatorname{SU}(2)$. By \cite[Proposition 10 in Chapter VII \S2]{BourbakiBookIntegrationII2004}, $H$ is unimodular too. We choose its bi-invariant Haar measure $\mu^H=\mu^H_R=\mu^H_L$ as the measure on $H$. From the construction of the homogeneous space, we see that $M=H\backslash G$ is a homogeneous space with a natural sub-Riemannian structure as described in Section~\ref{sec.SubRiemannianStructureonHomogoneousSpaces}.

We now study how the hypoelliptic logarithmic Sobolev constant of $M$ depends on its underlying geometry and the dimension.

\begin{theorem} \label{thm.LSIonHomogeneousSpaceIII}
There is a natural sub-Riemannian structure $\left(M,\mathcal{H}^{M},\langle\cdot,\cdot\rangle^{M}_{\mathcal{H}}\right)$ on $M$ induced by the transitive action by $\operatorname{SU}(2)\times \cdots\times \operatorname{SU}(2)$ and the logarithmic Sobolev inequality $LSI_C\left(M,\mathcal{H}^{M},\mathcal{D}\left(\mathcal{E}_{M}\right),\mu_t^{M}\right)$ holds with the constant $C\left(M,\mathcal{H}^{M},\mu_t^{M}\right)$. Moreover, we can choose
\begin{align*}
C\left(M,\mathcal{H}^{M},\mu_t^{M}\right)=C(\operatorname{SU}(2),t),
\end{align*}
which is the same as the constant for $\operatorname{SU}(2)$. Thus the logarithmic Sobolev constant $C\left(M,\mathcal{H}^{M},\mu_t^{M}\right)$ on the $M$ is independent of its dimension.
\end{theorem}
For the rest of this section, we consider some concrete examples of homogeneous spaces. We specify their isotropy subgroups as a closed Lie subgroup $H$ of $G$ and then identify them with a quotient space. This enables us to prove a dimension-independent logarithmic Sobolev inequality on these spaces.

\subsection{Hopf fibration} \label{sec.HopfFibration}

The Hopf fibration is another important example, for more details we refer to \cite{BaudoinDemniWangNotes2022}. We use the Hopf fibration  to describe the complex projective space $\mathbb{C}\mathbb{P}^1$ as a homogeneous space under the transitive action by $\operatorname{SU}(2)$. In the rest of this section, we identify this homogeneous space with a quotient space and present the corresponding hypoelliptic logarithmic Sobolev inequality.

Let $G=\operatorname{SU}(2)\cong \mathbb{S}^3$ and 
\[
H=\operatorname{U}(1)=
\left\{\begin{pmatrix}
e^{i\theta} & 0 \\ 0 & e^{-i\theta}
\end{pmatrix}:\theta\in \mathbb{R}\right\}=\operatorname{exp}\left(\operatorname{Span}\{Z\}\right)\cong \mathbb{S}^1.
\]  
Then $H$ is not a normal subgroup of $G$. The left action of $\operatorname{U}(1)$ on $\operatorname{SU}(2)$ 
\begin{align*}
& H\times G \rightarrow G
\\
&
\left(\begin{pmatrix}
e^{i\theta} & 0 \\ 0 & e^{-i\theta}
\end{pmatrix},\begin{pmatrix}
z_1 & z_2 \\
-\overline{z_2} & \overline{z_1}
\end{pmatrix}\right) \mapsto \begin{pmatrix}
e^{i\theta}z_1 & e^{-i\theta}z_2 \\
-e^{i\theta}\overline{z_2} & e^{-i\theta}\overline{z_1}
\end{pmatrix}
\end{align*}
is proper but not transitive. This action is isometric in the sub-Riemannian sense, that is, under the action $H$, an orthonormal frame is mapped to an orthonormal frame of $\left(\operatorname{SU}(2),\mathcal{H},\langle\cdot,\cdot\rangle_{\mathcal{H}}^{\operatorname{SU}(2)}\right)$.

The homogeneous space $M$ is no 
longer a Lie group but it is topologically isomorphic to the complex projective space $\mathbb{C}\mathbb{P}^1$. For simplicity, we will identify $M$ with $\mathbb{C}\mathbb{P}^1$ for the rest of this section. Also, we have the following smooth map
\begin{align}
& \varphi:\operatorname{SU}(2) \rightarrow \mathbb{C}\mathbb{P}^1\cong M \notag
\\
& \varphi(g):=\varphi(z_1, z_2)=\left(2z_1\overline{z_2},\vert z_1\vert^2-\vert z_2\vert^2\right). \notag
\end{align}
Moreover, $\left(\mathbb{C}\mathbb{P}^1,\mathcal{H}^{\mathbb{C}\mathbb{P}^1},\langle\cdot,\cdot\rangle^{\mathbb{C}\mathbb{P}^1}_{\mathcal{H}}\right)$ is a Riemannian manifold, which is a special case of sub-Riemannian manifolds, and $\Delta_{\mathcal{H}}^{\mathbb{C}\mathbb{P}^1}$ is the Laplacian on $M \cong \mathbb{C}\mathbb{P}^1$. 

In this case, let  $\mu_t^{\operatorname{SU}(2)}$ be the hypoelliptic heat kernel measure on $\operatorname{SU}(2)$ associated to the sub-Laplacian $\Delta_{\mathcal{H}}^{\operatorname{SU}(2)}$ and $\mu_t^{M}$ the heat kernel measure $\mu_t^{\mathbb{C}\mathbb{P}^1}$ on $\mathbb{C}\mathbb{P}^1$ associated to the Laplacian on $M \cong \mathbb{C}\mathbb{P}^1$. Applying Theorem~\ref{thm.LSIonHomogeneousSpaceIII}, we obtain the following result.  This result also implies \cite[Corollary 4.5]{Gross1992}.

\begin{proposition}
On the complex projective space $\mathbb{C}\mathbb{P}^1$, there is a natural sub-Riemannian structure $\left(\mathbb{C}\mathbb{P}^1,\mathcal{H}^{\mathbb{C}\mathbb{P}^1},\langle\cdot,\cdot\rangle^{\mathbb{C}\mathbb{P}^1}_{\mathcal{H}}\right)$ induced by the transitive action by $SU(2)$, which indeed coincides with the original Riemannian structure on $\mathbb{C}\mathbb{P}^1$. Then $\operatorname{LSI}_C\left(\mathbb{C}\mathbb{P}^1, \mathcal{H}, \mathcal{D}\left(\mathcal{E}_{\mathbb{C}\mathbb{P}^1}\right), \mu_t^{\mathbb{C}\mathbb{P}^1}\right)$ holds with the logarithmic Sobolev constant $C(\operatorname{SU}(2),t)$, which is the same as the logarithmic Sobolev constant on $\operatorname{SU}(2)$.
\end{proposition}

\subsection{$\operatorname{SO}(3)$} \label{sec.ExampleSO(3)}
The special orthogonal group $\operatorname{SO}(3)$ is a homogeneous space under the transitive action of $\operatorname{SU}(2)$. For the rest of this section, we describe this homogeneous space as a quotient space and present the corresponding hypoelliptic logarithmic Sobolev inequality. 

Consider $G=\operatorname{SU}(2)$ as in the previous section and $H=\mathbb{Z}_2$. The homogeneous space $M$ is a Lie group and it is isomorphic to $\operatorname{SO}(3)$. For simplicity, we will identify $M$ with $\operatorname{SO}(3)$ for the rest of this section. Also, $\mathfrak{su}(2)$ is isomorphic to the Lie algebra $\mathfrak{so}(3)$ of $\operatorname{SO}(3)$ (see \cite[Example 3.29]{HallLieBook}) and there exists a two-to-one and onto Lie group homomorphism $\varphi:\operatorname{SU}(2) \rightarrow \operatorname{SO}(3) \cong M$ (see \cite[Proposition 1.19]{HallLieBook}). This means that $\operatorname{SU}(2)$ is a universal cover of $\operatorname{SO}(3)$.

Using the approach in Section~\ref{sec.SubRiemannianStructureonHomogoneousSpaces}, we see that there is a natural sub-Riemannian structure $\left(\operatorname{SO}(3),\mathcal{H}^{\operatorname{SO}(3)},
\langle\cdot,\cdot\rangle^{\operatorname{SO}(3)}_{\mathcal{H}}\right)$ on $\operatorname{SO}(3)$. For more details, we refer to \cite[Section 7.7.2]{AgrachevBarilariBoscainBook2020}. Let  $\mu_t^{\operatorname{SU}(2)}$ be the hypoelliptic heat kernel measure on $\operatorname{SU}(2)$ associated to the sub-Laplacian $\Delta_{\mathcal{H}}^{\operatorname{SU}(2)}$ and $\mu_t^{M}$ the hypoelliptic heat kernel measure $\mu_t^{\operatorname{SO}(3)}$ on $\operatorname{SO}(3)$ associated to the sub-Laplacian $\Delta_{\mathcal{H}}^{\operatorname{SO}(3)}$ on $\operatorname{SO}(3)$.
Then Theorem~\ref{thm.LSIonHomogeneousSpaceIII} implies the following result. 

\begin{proposition}
The logarithmic Sobolev inequality holds on the sub-Riemannian manifold $\left(\operatorname{SO}(3),\mathcal{H}^{\operatorname{SO}(3)},
\langle\cdot,\cdot\rangle^{\operatorname{SO}(3)}_{\mathcal{H}}\right)$ with the logarithmic Sobolev constant $C(\operatorname{SU}(2),t)$, which is the same as the logarithmic Sobolev constant on $\operatorname{SU}(2)$.
\end{proposition}

\subsection{$\operatorname{SO}(4)$} \label{sec.ExampleSO(4)}

The special orthogonal group $\operatorname{SO}(4)$ is a homogeneous space under the transitive action of $\operatorname{SU}(2)\times \operatorname{SU}(2)$. In the rest of this section, we describe this homogeneous space as a quotient space and present the corresponding hypoelliptic logarithmic Sobolev inequality. 

Consider $G=\operatorname{SU}(2)\times \operatorname{SU}(2)$ and $H=\mathbb{Z}_2$, then the homogeneous space $M$ is a Lie group and it is isomorphic to $\operatorname{SO}(4)$. This means that $\operatorname{SU}(2)\times \operatorname{SU}(2)$ is a universal cover of $\operatorname{SO}(4)$. For simplicity, we will identify $M$ with $\operatorname{SO}(4)$. From Section~\ref{sec.SubRiemannianStructureonHomogoneousSpaces}, there is a natural sub-Riemannian structure $\left(\operatorname{SO}(4),\mathcal{H}^{\operatorname{SO}(4)},\langle\cdot,\cdot\rangle^{\operatorname{SO}(4)}_{\mathcal{H}}\right)$ on $\operatorname{SO}(4)$.

In this case, $\mu_t^M$ is the hypoelliptic heat kernel measure $\mu_t^{\operatorname{SO}(4)}$ on $\operatorname{SO}(4)$ associated to the sub-Laplacian $\Delta_{\mathcal{H}}^{\operatorname{SO}(4)}$ on $\operatorname{SO}(4)$. Applying Theorem~\ref{thm.LSIonHomogeneousSpaceIII}, we  obtain the following result.

\begin{proposition}
There is a natural sub-Riemannian structure on $\operatorname{SO}(4)$ induced by the transitive action by $\operatorname{SU}(2) \times \operatorname{SU}(2)$ and the logarithmic Sobolev inequality holds with the logarithmic Sobolev constant $C(\operatorname{SU}(2),t)$, which is the same as the logarithmic Sobolev constant on $\operatorname{SU}(2)$.
\end{proposition}

\begin{acknowledgement}
The  authors would like to thank Fabrice~Baudoin and Tai~Melcher for helpful discussions during the preparation of this work.
\end{acknowledgement}

\bibliographystyle{amsplain}

\begin{thebibliography}{10}

\bibitem{AgrachevBarilariBoscainBook2020}
Andrei Agrachev, Davide Barilari, and Ugo Boscain, \emph{A comprehensive
  introduction to sub-{R}iemannian geometry}, Cambridge Studies in Advanced
  Mathematics, vol. 181, Cambridge University Press, Cambridge, 2020, From the
  Hamiltonian viewpoint, With an appendix by Igor Zelenko. \MR{3971262}

\bibitem{BakryBaudoinBonnefontChafai2008}
Dominique Bakry, Fabrice Baudoin, Michel Bonnefont, and Djalil Chafa{\"{\i}},
  \emph{On gradient bounds for the heat kernel on the {H}eisenberg group}, J.
  Funct. Anal. \textbf{255} (2008), no.~8, 1905--1938. \MR{2462581
  (2010m:35534)}

\bibitem{BakryGentilLedouxBook}
Dominique Bakry, Ivan Gentil, and Michel Ledoux, \emph{Analysis and geometry of
  {M}arkov diffusion operators}, Grundlehren der Mathematischen Wissenschaften
  [Fundamental Principles of Mathematical Sciences], vol. 348, Springer, Cham,
  2014. \MR{3155209}

\bibitem{BaudoinBonnefont2012}
Fabrice Baudoin and Michel Bonnefont, \emph{Log-{S}obolev inequalities for
  subelliptic operators satisfying a generalized curvature dimension
  inequality}, J. Funct. Anal. \textbf{262} (2012), no.~6, 2646--2676.
  \MR{2885961}

\bibitem{BaudoinDemniWangNotes2022}
Fabrice Baudoin, Nizar Demni, and Jing Wang, \emph{Stochastic areas, horizontal
  brownian motions, and hypoelliptic heat kernels}, 2022, arxiv preprint.

\bibitem{BaudoinGarofalo2017}
Fabrice Baudoin and Nicola Garofalo, \emph{Curvature-dimension inequalities and
  {R}icci lower bounds for sub-{R}iemannian manifolds with transverse
  symmetries}, J. Eur. Math. Soc. (JEMS) \textbf{19} (2017), no.~1, 151--219.
  \MR{3584561}

\bibitem{BaudoinGordinaSarkar2023}
Fabrice Baudoin, Maria Gordina, and R~Sarkar, \emph{Tbd}, 2023..., In progress.

\bibitem{BaudoinHairerTeichmann2008}
Fabrice Baudoin, Martin Hairer, and Josef Teichmann, \emph{Ornstein-{U}hlenbeck
  processes on {L}ie groups}, J. Funct. Anal. \textbf{255} (2008), no.~4,
  877--890. \MR{2433956}

\bibitem{BonfiglioliLanconelliUguzzoniBook}
A.~Bonfiglioli, E.~Lanconelli, and F.~Uguzzoni, \emph{Stratified {L}ie groups
  and potential theory for their sub-{L}aplacians}, Springer Monographs in
  Mathematics, Springer, Berlin, 2007. \MR{2363343}

\bibitem{BonnefontChafaiHerry2020}
Michel Bonnefont, Djalil Chafa\"{\i}, and Ronan Herry, \emph{On logarithmic
  {S}obolev inequalities for the heat kernel on the {H}eisenberg group}, Ann.
  Fac. Sci. Toulouse Math. (6) \textbf{29} (2020), no.~2, 335--355.
  \MR{4150544}

\bibitem{BourbakiBookIntegrationI2004}
Nicolas Bourbaki, \emph{Integration. {I}. {C}hapters 1--6}, Elements of
  Mathematics (Berlin), Springer-Verlag, Berlin, 2004, Translated from the
  1959, 1965 and 1967 French originals by Sterling K. Berberian. \MR{2018901}

\bibitem{BourbakiBookIntegrationII2004}
\bysame, \emph{Integration. {II}. {C}hapters 7--9}, Elements of Mathematics
  (Berlin), Springer-Verlag, Berlin, 2004, Translated from the 1963 and 1969
  French originals by Sterling K. Berberian. \MR{2098271}

\bibitem{CalinChangFurutaniIwasakiBook2011}
Ovidiu Calin, Der-Chen Chang, Kenro Furutani, and Chisato Iwasaki, \emph{Heat
  kernels for elliptic and sub-elliptic operators}, Applied and Numerical
  Harmonic Analysis, Birkh\"auser/Springer, New York, 2011, Methods and
  techniques. \MR{2723056 (2011i:58037)}

\bibitem{Caratheodory1909}
C.~Carath\'{e}odory, \emph{Untersuchungen \"{u}ber die {G}rundlagen der
  {T}hermodynamik}, Math. Ann. \textbf{67} (1909), no.~3, 355--386.
  \MR{1511534}

\bibitem{DagherZegarlinski2022a}
Esther~Bou Dagher and Bogus\l~aw Zegarli\'{n}ski, \emph{Coercive inequalities
  in higher-dimensional anisotropic {H}eisenberg group}, Anal. Math. Phys.
  \textbf{12} (2022), no.~1, Paper No. 3, 33. \MR{4334243}

\bibitem{DriverGrossSaloff-Coste2009a}
Bruce~K. Driver, Leonard Gross, and Laurent Saloff-Coste, \emph{Holomorphic
  functions and subelliptic heat kernels over {L}ie groups}, J. Eur. Math. Soc.
  (JEMS) \textbf{11} (2009), no.~5, 941--978. \MR{2538496 (2010h:32052)}

\bibitem{DriverGrossSaloff-Coste2010}
\bysame, \emph{Growth of {T}aylor coefficients over complex homogeneous
  spaces}, Tohoku Math. J. (2) \textbf{62} (2010), no.~3, 427--474.
  \MR{2742018}

\bibitem{DriverMelcher2005}
Bruce~K. Driver and Tai Melcher, \emph{Hypoelliptic heat kernel inequalities on
  the {H}eisenberg group}, J. Funct. Anal. \textbf{221} (2005), 340--365.

\bibitem{Eldredge2010}
Nathaniel Eldredge, \emph{Gradient estimates for the subelliptic heat kernel on
  {$H$}-type groups}, J. Funct. Anal. \textbf{258} (2010), no.~2, 504--533.
  \MR{2557945 (2011d:35217)}

\bibitem{FeffermanSanchez-Calle1986}
Charles~L. Fefferman and Antonio S\'anchez-Calle, \emph{Fundamental solutions
  for second order subelliptic operators}, Ann. of Math. (2) \textbf{124}
  (1986), no.~2, 247--272. \MR{855295}

\bibitem{Folland2004}
G.~B. Folland, \emph{Compact {H}eisenberg manifolds as {CR} manifolds}, J.
  Geom. Anal. \textbf{14} (2004), no.~3, 521--532. \MR{2077163}

\bibitem{FollandHABook}
Gerald~B. Folland, \emph{A course in abstract harmonic analysis}, Studies in
  Advanced Mathematics, CRC Press, Boca Raton, FL, 1995. \MR{MR1397028
  (98c:43001)}

\bibitem{FrankLieb2012}
Rupert~L. Frank and Elliott~H. Lieb, \emph{Sharp constants in several
  inequalities on the {H}eisenberg group}, Ann. of Math. (2) \textbf{176}
  (2012), no.~1, 349--381. \MR{2925386}

\bibitem{FukushimaOshimaTakedaBook2011}
Masatoshi Fukushima, Yoichi Oshima, and Masayoshi Takeda, \emph{Dirichlet forms
  and symmetric {M}arkov processes}, extended ed., De Gruyter Studies in
  Mathematics, vol.~19, Walter de Gruyter \& Co., Berlin, 2011. \MR{2778606}

\bibitem{GalloneMichelangeliPozzoli2019}
Matteo Gallone, Alessandro Michelangeli, and Eugenio Pozzoli, \emph{On
  geometric quantum confinement in {G}rushin-type manifolds}, Z. Angew. Math.
  Phys. \textbf{70} (2019), no.~6, Paper No. 158, 17. \MR{4019735}

\bibitem{GorbatsevichOnishchikVinbergBook1997}
V.~V. Gorbatsevich, A.~L. Onishchik, and E.~B. Vinberg, \emph{Foundations of
  {L}ie theory and {L}ie transformation groups}, Springer-Verlag, Berlin, 1997,
  Translated from the Russian by A. Kozlowski, Reprint of the 1993 translation
  [{\it Lie groups and Lie algebras. I}, Encyclopaedia Math. Sci., 20,
  Springer, Berlin, 1993; MR1306737 (95f:22001)]. \MR{1631937}

\bibitem{GordinaLaetsch2016a}
Maria Gordina and Thomas Laetsch, \emph{Sub-{L}aplacians on {S}ub-{R}iemannian
  {M}anifolds}, Potential Anal. \textbf{44} (2016), no.~4, 811--837.
  \MR{3490551}

\bibitem{GordinaLuo2022}
Maria Gordina and Liangbing Luo, \emph{Logarithmic {S}obolev inequalities on
  non-isotropic {H}eisenberg groups}, J. Funct. Anal. \textbf{283} (2022),
  no.~2, Paper No. 109500. \MR{4410358}

\bibitem{GordonWilson1986}
Carolyn~S. Gordon and Edward~N. Wilson, \emph{The spectrum of the {L}aplacian
  on {R}iemannian {H}eisenberg manifolds}, Michigan Math. J. \textbf{33}
  (1986), no.~2, 253--271. \MR{837583}

\bibitem{Gross1975c}
Leonard Gross, \emph{Logarithmic {S}obolev inequalities}, Amer. J. Math.
  \textbf{97} (1975), no.~4, 1061--1083. \MR{MR0420249 (54 \#8263)}

\bibitem{Gross1992}
\bysame, \emph{Logarithmic {S}obolev inequalities on {L}ie groups}, Illinois J.
  Math. \textbf{36} (1992), no.~3, 447--490. \MR{1161977 (93i:22012)}

\bibitem{HallLieBook}
Brian~C. Hall, \emph{Lie groups, {L}ie algebras, and representations}, Graduate
  Texts in Mathematics, vol. 222, Springer-Verlag, New York, 2003, An
  elementary introduction. \MR{1997306 (2004i:22001)}

\bibitem{HebischZegarlinski2010}
W.~Hebisch and B.~Zegarli\'nski, \emph{Coercive inequalities on metric measure
  spaces}, J. Funct. Anal. \textbf{258} (2010), no.~3, 814--851. \MR{2558178}

\bibitem{Hormander1967a}
Lars H{\"o}rmander, \emph{Hypoelliptic second order differential equations},
  Acta Math. \textbf{119} (1967), 147--171. \MR{0222474 (36 \#5526)}

\bibitem{HughenThesis1995}
Walker~Keener Hughen, \emph{The sub-{R}iemannian geometry of three-manifolds},
  ProQuest LLC, Ann Arbor, MI, 1995, Thesis (Ph.D.)--Duke University.
  \MR{2692648}

\bibitem{Hunt1956a}
G.~A. Hunt, \emph{Semi-groups of measures on {L}ie groups}, Trans. Amer. Math.
  Soc. \textbf{81} (1956), 264--293. \MR{MR0079232 (18,54a)}

\bibitem{JerisonSanchez-Calle1987}
David Jerison and Antonio S\'anchez-Calle, \emph{Subelliptic, second order
  differential operators}, Complex analysis, {III} ({C}ollege {P}ark, {M}d.,
  1985--86), Lecture Notes in Math., vol. 1277, Springer, Berlin, 1987,
  pp.~46--77. \MR{922334}

\bibitem{KnappBook1996}
Anthony~W. Knapp, \emph{Lie groups beyond an introduction}, Progress in
  Mathematics, vol. 140, Birkh\"auser Boston Inc., Boston, MA, 1996.
  \MR{1399083 (98b:22002)}

\bibitem{LeeBook2003SmoothManifold}
John~M. Lee, \emph{Introduction to smooth manifolds}, second ed., Graduate
  Texts in Mathematics, vol. 218, Springer, New York, 2013. \MR{2954043}

\bibitem{LiHong-Quan2006}
Hong-Quan Li, \emph{Estimation optimale du gradient du semi-groupe de la
  chaleur sur le groupe de {H}eisenberg}, J. Funct. Anal. \textbf{236} (2006),
  no.~2, 369--394. \MR{MR2240167 (2007d:58045)}

\bibitem{Luo2023}
Liangbing Luo, \emph{Poincar\'e inequalities on homogeneous spaces}, 2023...,
  In progress.

\bibitem{Lust-Piquard2010}
Fran\c{c}oise Lust-Piquard, \emph{Ornstein-{U}hlenbeck semi-groups on
  stratified groups}, J. Funct. Anal. \textbf{258} (2010), no.~6, 1883--1908.
  \MR{2578458}

\bibitem{MontgomeryBook2002}
Richard Montgomery, \emph{A tour of subriemannian geometries, their geodesics
  and applications}, Mathematical Surveys and Monographs, vol.~91, American
  Mathematical Society, Providence, RI, 2002. \MR{1867362 (2002m:53045)}

\bibitem{PrandiRizziSeri2018}
Dario Prandi, Luca Rizzi, and Marcello Seri, \emph{Quantum confinement on
  non-complete {R}iemannian manifolds}, J. Spectr. Theory \textbf{8} (2018),
  no.~4, 1221--1280. \MR{3870067}

\bibitem{SagleWaldeBook1973}
Arthur~A. Sagle and Ralph~E. Walde, \emph{Introduction to {L}ie groups and
  {L}ie algebras}, Pure and Applied Mathematics, vol. Vol. 51, Academic Press,
  New York-London, 1973. \MR{360927}

\bibitem{Schechtman2003a}
Gideon Schechtman, \emph{Concentration results and applications}, Handbook of
  the geometry of {B}anach spaces, {V}ol. 2, North-Holland, Amsterdam, 2003,
  pp.~1603--1634. \MR{1999604}

\bibitem{Strichartz1986}
Robert~S. Strichartz, \emph{Sub-{R}iemannian geometry}, J. Differential Geom.
  \textbf{24} (1986), no.~2, 221--263. \MR{862049}

\bibitem{Varopoulos1988a}
N.~Th. Varopoulos, \emph{Analysis on {L}ie groups}, J. Funct. Anal. \textbf{76}
  (1988), no.~2, 346--410. \MR{924464}

\bibitem{Zhang-Y2021note}
Ye~Zhang, \emph{A note on gradient estimates for the heat semigroup on
  nonisotropic heisenberg groups}, 2021, arxiv preprint.

\end{thebibliography}
\providecommand{\bysame}{\leavevmode\hbox to3em{\hrulefill}\thinspace}
\providecommand{\MR}{\relax\ifhmode\unskip\space\fi MR }
\providecommand{\MRhref}[2]{%
  \href{http://www.ams.org/mathscinet-getitem?mr=#1}{#2}
}
\providecommand{\href}[2]{#2}

\end{document}